\documentclass[11pt]{article}
\usepackage{amsfonts, amsmath, amssymb, amscd, amsthm, color, graphicx, mathrsfs, wasysym, setspace, mdwlist, float}
\newcommand{\calA}{\mathcal A}

\newcommand{\ZZ}{\mathbb Z}
\newcommand{\NN}{{\mathbb N}}

\newcommand{\Lab}{{\bf Lab}}
\newcommand{\CL}{\operatorname{CL}}

\newcommand{\Cont}{\operatorname{Cont}}

\newtheorem{thm}{Theorem}[section]
\newtheorem{cor}[thm]{Corollary}
\newtheorem{lem}[thm]{Lemma}

\newtheorem{prop}[thm]{Proposition}

\theoremstyle{definition}
\newtheorem{defn}[thm]{Definition}

\theoremstyle{remark}
\newtheorem{rem}[thm]{Remark}
\newtheorem{ex}[thm]{Example}
\newtheorem{claim}[thm]{Claim}
\newtheorem{conv}[thm]{Convention}

\usepackage{hyperref}
\hypersetup{linktocpage}
\usepackage{xcolor}
\hypersetup{colorlinks,
    linkcolor={red!50!black},
    citecolor={blue!80!black},
    urlcolor={blue!80!black}}
\title{Conjugator length in finitely generated groups}
\author{G. Goffer, M. Mihaila, D. Osin}
\date{}

\begin{document}

\maketitle

\vspace{-3mm}
\begin{abstract}
We describe all functions $\mathbb{N}\rightarrow \mathbb{N}$ that can be realized, up to the standard equivalence, as conjugator length functions of finitely generated groups. Furthermore, we show that any two increasing functions $f,g\colon \NN\to \NN$ can be simultaneously realized as conjugator length functions of finitely generated, commensurable  (in particular, quasi-isometric) groups. 
\end{abstract}

\section{Introduction}

The interplay between algorithmic complexity and asymptotic invariants of finitely generated groups is one of the central themes in geometric group theory. Prominent examples arising in this way include Dehn functions, subgroup distortion, and many others (see~\cite{Gro}). When studying such invariants, two fundamental questions arise:
\begin{itemize}
    \item[(a)] What values can the invariants attain?
    
    \item[(b)] How do these values behave under natural equivalence relations on the set of finitely generated groups? 
\end{itemize}

In this paper, we address these questions for the conjugator length function, an asymptotic invariant closely related to the conjugacy problem that has attracted significant attention in recent years. At the final stage of our work, we saw the preprint \cite{Van26} and learned that Vandeputte had independently obtained results that cover Theorem~1.3 and Proposition~6.1 of this paper. We therefore make no claim of priority for these overlapping results. Nevertheless, our proofs use different methods, and Theorem~1.4 is not covered by \cite{Van26}; moreover, its proof relies heavily on the construction developed in the proof of Theorem~1.3. Thus, we believe that our paper may still be of interest. Our results were obtained independently, and a preprint version of the paper circulated among experts in the summer of 2026; it was also cited in Section~5 of \cite{BRS}. For these reasons, we have decided to post the paper in its current form.

We begin by recalling the definition.  Let $G$ be a group generated by a finite set $X$. For $g\in G$, we denote by $|g|_X$ the word length of $g$ with respect to $X$.

\begin{defn}
   Given two conjugate elements $a,b\in G$, let $c_X(a,b)$ be the length of a shortest element of $G$ conjugating $a$ to $b$; that is,
    $$
    c_X(a,b)=\min \big\{ |g|_X\mid g\in G,\; g^{-1}ag=b\big\}.
    $$
    Further, for each $n\in \NN$, let $$\CL_{G,X}(n)=\max \{ c_X(a,b)\},$$ where the maximum is taken over all pairs of conjugate elements $a,b\in G$ of total length $|a|_X+|b|_X\le n$. The map $\CL_{G,X}\colon \NN\to \NN\cup\{0\}$ is called the \emph{conjugator length function} of $G$ with respect to  $X$.
\end{defn}

Recall that the \textit{word problem} for the group $G$ is to decide whether a given word over the alphabet $X^{\pm 1}$ represents the trivial element. Similarly, the \textit{conjugacy problem} asks whether two given words represent conjugate elements in $G$. It is straightforward to verify that, for any finitely generated group $G$ with decidable word problem, the conjugacy problem is decidable if and only if $\CL_{G,X}$ is computable (or, equivalently, bounded by a computable function from above). 

Given this close relationship, it is not surprising that the study of the conjugator length function, at least implicitly, has been present in geometric and combinatorial group theory for a long time. Indeed, computing upper bounds for $\CL_{G,X}$ often goes hand in hand with solving the conjugacy problem in $G$; see, for example, \cite{G60} and \cite{L90}, where explicit linear upper bounds on conjugator length are obtained for small cancellation and hyperbolic groups, respectively. However, the first formal definition of $\CL_{G,X}$ appears to be given by Brick and Corson in \cite{BC}, where it is introduced under the name of the \textit{annular width function}. Since then, conjugator length functions have attracted significant attention and have been computed for many important classes of groups \cite{BD14,BM22, MM00,S15a,S15b,S16a,S16b,T13}.

Similarly to other asymptotic invariants, conjugator length functions are usually studied up to a natural equivalence. 

\begin{defn}
    For two functions $f,g\colon \NN\to \NN\cup\{0\}$, we write $f\preccurlyeq g$ if there exist $C,D\in \NN$ such that
\[
f(n)\le C g(Dn)\qquad \forall\, n\in \NN.
\]
Further, we denote by $\sim $ the equivalence relation on ${\NN\cup\{0\}}^\NN$ induced by the quasi-order $\preccurlyeq$. That is, we have $f\sim g$, if and only if $f\preccurlyeq g$ and $g \preccurlyeq f$. 
\end{defn}

It is straightforward to verify that $\CL_{G,X}\sim \CL_{G,Y}$ for any finite generating sets $X$ and $Y$ of a group $G$ (see \cite[Theorem 1.2]{BC}). In what follows, we often consider the conjugator length function up to equivalence and omit the subscript $X$ from the notation.

The question of which functions can be realized, up to equivalence, as $\CL_{G}$ for a finitely generated group $G$ has recently attracted significant attention. In \cite{BR25b,BR25a}, Bridson and Riley showed that all polynomial functions $n^d$, $d\in \mathbb{N}$, occur as conjugator length functions of nilpotent groups. In \cite{BR1225}, they further constructed a family of finitely presented groups with conjugator length functions $n^\alpha$ for $\alpha$ in a dense subset of $[2,\infty)$. 

An ultimate result in this direction was recently obtained in \cite{GW}, where Gillis and Wagner proved that all time functions of $S$-machines can be realized, up to a quadratic summand, as conjugator length functions of finitely presented groups. As a very particular case, this implies that, for any algebraic real number $\alpha \ge 2$, as well as for $\alpha = \pi$ and $\alpha = e$, there exists a finitely presented group $G$ such that $\CL_G(n) \sim n^\alpha$. Furthermore, all functions realizable as Dehn functions can also be realized as conjugator length functions of finitely presented groups \cite[Theorem 1.4]{GW}.

This left open the question of whether a superlinear subquadratic function can be realized as $\CL_G$ for a finitely presented (or just finitely generated) group $G$. These problems were explicitly stated in \cite{GW}; shortly after, the question for finitely generated groups was answered by Vandeputte \cite{Van} affirmatively. Furthermore, \cite[Theorem~C]{Van} shows that a wide class of functions satisfying certain growth and regularity conditions (e.g., any function $f\colon \NN\to \NN\cup\{0\}$ such that $f(n)\ge n^3$ and $f(3n)\ge 3f(n)$) can be realized as $\CL_G$ for a finitely generated solvable group of derived length $3$. 

Despite these advances, the general realization problem for finitely generated groups remained open. Our first result provides a complete solution. As mentioned above, it was independently proved by Vandeputte \cite{Van26}. Below, we say that a function $f\colon \NN\to \NN\cup\{0\}$ is \textit{at least linear} if $n\preccurlyeq f(n)$.

\begin{thm}\label{Thm:main1}
For any group $G$ and any finite generating set $X$ of $G$, the function $\CL_{G,X}$ is non-decreasing and either bounded or at least linear. Conversely, for any non-decreasing function $f \colon \NN\to \NN\cup\{0\}$ that is either bounded or at least linear, there exists a finitely generated group $G$ such that $\CL_G(n)\sim f(n)$; moreover, if $f$ is computable, the group $G$ can be chosen recursively presented.
\end{thm}

It is easy to see that the conjugator length function of a finitely generated group $G$ is bounded precisely when the quotient $G/Z(G)$, of $G$ by its center, is finite. Although this is an elementary observation, we could not find its proof in the literature; for completeness, we provide it in Section \ref{sec:bounded CL} (see Proposition \ref{Prop:BCL}).

Recall that two groups are \textit{commensurable} if they contain isomorphic subgroups of finite index. For finitely generated groups, commensurability is an important instance of the more general notion of quasi-isometry. Many asymptotic invariants traditionally studied in geometric group theory exhibit a certain level of stability under these relations. In contrast, invariants related to conjugacy tend to behave rather poorly, as illustrated by the Collins--Miller result \cite{CM77}, showing that solvability of the conjugacy problem in a finitely presented group need not pass to finite index subgroups or finite extensions, and by the extreme non-invariance of the conjugacy growth function established in \cite{HO}.

Thus, one expects the conjugacy length function to exhibit similarly pathological behavior. Our next theorem confirms this expectation in its strongest possible form.

\begin{thm}\label{Thm:main2}
For any non-decreasing, at least linear functions $f,g \colon \NN\to \NN\cup\{ 0\}$, there exist commensurable (in particular, quasi-isometric) finitely generated groups $A$ and $B$ such that $\CL_A(n)\sim f(n)$ and $\CL_B(n)\sim g(n)$.
\end{thm}

In the proof of Theorems \ref{Thm:main1} and \ref{Thm:main2}, the required groups are defined by  \textit{bilateral presentations} (see Definition \ref{Defn:bi}) generalizing the standard presentations of HNN-extensions. Along the way, we prove several results about such presentations, which appear to be of independent interest.

The paper is organized as follows. In the next section, we recall some standard notation and facts about van Kampen diagrams used in the paper. In Section \ref{sec:bilateral pres}, we introduce bilateral presentations and obtain some general facts about van Kampen diagrams over them. The proofs of Theorems \ref{Thm:main1} and \ref{Thm:main2} are given in Sections \ref{sec:proof of main1} and \ref{sec:proof of main2} respectively. Finally, in Section \ref{sec:bounded CL}, we discuss groups with bounded conjugator length. 

\paragraph{Acknowledgments.} The authors would like to thank Be'eri Greenfeld for helpful comments and discussions during his initial involvement in the early stages of this work. The third-named author was supported by the NSF grant DMS-2405032 and the Simons Fellowship in Mathematics MP-SFM-00005907.

\section{Preliminaries}\label{sec:prel}

\paragraph{2.1. Notation.} We begin by recalling the standard terminology and notation used throughout the paper. 

Let $G$ be a group generated by a set $X$. We denote by $F(X)$ the free group on $X$ and use the standard notation $X^{-1}=\{ x^{-1}\mid x\in X\}$ and $X^{\pm 1}=X\sqcup X^{-1}$. Let $U$, $W$ be words over the alphabet $X^{\pm 1}$. We denote by $\|U\|$ the length of $U$ (that is, the number of letters in $U$); further, we write $U \equiv W$ to indicate that two words $U$, $W$ over the alphabet $X^{\pm 1}$ are equal as strings of letters, and $U=_G W$ to indicate that they represent the same element of the group $G$. The concatenation of two words $W_1$ and $W_2$  over the alphabet $X^{\pm 1}$ is denoted by $W_1W_2$. A word $U$ is a \emph{cyclic shift} of $W$ if $U\equiv W_2W_1$ and $W\equiv W_1W_2$, for some subwords $W_1,W_2$ of $W$. 

For an element $g\in G$, we let $|g|_X$ denote the norm of $g$, namely, the length of a shortest word over $X$ representing $g$ in $G$. By abuse of notation, we sometimes refer to words over $X$ as elements of $G$ if there is no concern of ambiguities. A word $W$ over the alphabet $X^{\pm 1}$ is called \emph{reduced} if it does not contain a subword of the form $xx^{-1}$ or $x^{-1}x$, where $x\in X$. Further, $W$ is said to be \emph{geodesic in $G$} if $|W|_X=\|W\|$.

\paragraph{2.2. Van Kampen diagrams.}
We now briefly review the standard terminology and results on van Kampen diagrams necessary for our paper. For more details, we refer to  Chapter 5 of \cite{LS}.

In what follows, we refer to $0$--, $1$--, and $2$--dimensional cells of $2$-dimensional complexes as \textit{vertices, edges,} and \textit{faces}, respectively. Recall that a \emph{van Kampen diagram} $\Delta$ over a group presentation 
\begin{align}\label{eq:pres S|R}
    G=\langle \, X\mid \cal{R}\,\rangle
\end{align} is a finite, oriented, connected, planar $2$-complex embedded in the Euclidean plane and endowed with a labeling function $\Lab\colon E(\Delta) \to X^{\pm 1}$, such that $\Lab(e^{-1})=(\Lab(e))^{-1}$ for all $e\in E(\Delta)$, where $E(\Delta)$ denotes the set of oriented edges of $\Delta$.  In this paper, we will only consider disk or annular van Kampen diagrams. Recall that $\Delta$ is a \emph{disk diagram} if it is simply connected, and \emph{annular} if its complement in the plane consists of two connected components.

For a (combinatorial) path $p$ in $\Delta$, we denote by $\ell(p)$ its length, by $\Lab(p)$ its label, and by $p_-$, $p_+$ its starting and ending vertex, respectively. Given a face $\pi$ of $\Delta$, we denote by $\partial \pi$ the \textit{boundary} of $\pi$. Similarly, $\partial \Delta$ denotes the boundary of $\Delta$. If $\Delta $ is a disk diagram, we can think of $\partial \Delta$ as a combinatorial loop by picking a base vertex on $\partial \Delta$. Similarly, if $\Delta $ is annular, $\partial \Delta$ naturally decomposes into a (not necessarily disjoint) union of two loops, $\partial_{ext}\Delta $ and $\partial_{int}\Delta$, called the \textit{external and internal components of $\partial \Delta$}; as a set, $\partial_{ext}\Delta $ (respectively, $\partial_{int}\Delta$) can be defined as the intersections of $\partial \Delta$ with the unbounded (respectively, bounded) component of the complement of $\Delta$ in the plane. 

\begin{conv}\label{con:orientation}
We agree to read the labels of boundaries of faces and boundary components of $\Delta $ in the clockwise direction. These labels are defined up to a cyclic permutation, unless the starting point is explicitly mentioned; for example, we write $\Lab (\partial \pi)\equiv w$ for a face $\pi$ in $\Delta$ to mean that the label is equal to a cyclic permutation of a word $w$.
\end{conv}

For $\Delta$ to be a van Kampen diagram over (\ref{eq:pres S|R}), it is additionally required that  for any face $\pi$ of $\Delta$, we have 
$$
\Lab(\partial \pi)\equiv R^{\pm 1}, \;\;\;{\rm for\; some}\; R\in \mathcal{R}.
$$

We will use the following fundamental results on van Kampen diagrams (for details, see Sections 1 and 5 in Chapter 5 of \cite{LS} or Section 11 in \cite{OlshanskiiBook}). 

\begin{lem}\label{Lem:van Kampen} Let $G$ be a group given by the presentation (\ref{eq:pres S|R}). 
\begin{enumerate}
    \item[(a)] (Van Kampen Lemma)\; A non-empty word $W$ over the alphabet $X^{\pm 1}$ represents the identity in $G$ if and only if there exists a disk van Kampen diagram $\Delta$ over (\ref{eq:pres S|R}) with $\Lab (\partial\Delta)\equiv W$.
    \item[(b)] (Schupp Lemma) Non-empty words $U$ and $V$ over $X^{\pm 1}$ represent conjugate elements in $G$ if and only if there is an annular diagram $\Delta$ over (\ref{eq:pres S|R}) such that $\Lab(\partial_{ext} \Delta)\equiv U$ and $\Lab(\partial_{int}\Delta) \equiv V$.
\end{enumerate}
\end{lem}

\section{Bilateral group presentations}\label{sec:bilateral pres}

The analysis of van Kampen diagrams in the proofs of our main results relies on a generalization of certain ideas originally developed for HNN-extensions in \cite{MS}. In this paper, we work with a broader class of groups defined by ``HNN-like'' presentations. The purpose of this section is to prove several basic results about such groups, which will be used throughout the rest of the paper.

\paragraph{3.1. Bilateral presentations and geometry of van Kampen diagrams.} 
We begin by introducing the terminology.

\begin{defn}[\bf $T$-bilateral presentations]\label{Defn:bi}
Let $X$ and $T$ be disjoint sets. We say that a group presentation 
\begin{equation}\label{Eq:bi}
G=\langle X\cup T \mid \mathcal R \rangle 
\end{equation}
is \textit{$T$-bilateral} if every relator from $\mathcal R$ contains either no or exactly two letters from $T^{\pm 1}$.
\end{defn}

\begin{ex}\label{Ex:HNN}
    Let $G_0=\langle X\mid \mathcal R_0\rangle $ be an arbitrary group and let $A=\langle a_1, a_2, \ldots\rangle $ and $B=\langle b_1, b_2, \ldots \rangle $ be subgroups of $G_0$. For each $i$, we fix a word $A_i$ (respectively, $B_i$) over $X^{\pm 1}$ representing the element $a_i$ (respectively, $b_i$).  Suppose that there exists an isomorphism $\phi\colon A\to B$ such that $\phi(a_i)=b_i$ for all $i$. The associated \textit{HNN-extension} of the group $G_0$ is given by the presentation
\begin{equation}\label{Eq:HNN}
G=\langle\, X\cup\{t\}\mid \mathcal R_0,\; tA_it^{-1}=B_i, i=1,2, \ldots\,\rangle.
\end{equation} 
Clearly, this presentation is $\{ t\}$-bilateral. 
\end{ex} 

Let $\Delta$ be a van Kampen diagram over a $T$-bilateral presentation (\ref{Eq:bi}). By a \textit{$T$-edge} of $\Delta$ we will mean an edge labeled by $t^{\pm 1}$ for some $t\in T$.  Further, if $\pi$ is a face of $\Delta$ containing a $T$-edge, we say that $\pi $ is a $T$-\emph{face}. 

\begin{defn}[\bf $T$-bands and $T$-annuli]\label{def:t-bands}
A \textit{$T$-band} over (\ref{Eq:bi}) is a van Kampen diagram $\Sigma$ consisting of a sequence of pairwise distinct $T$-faces \begin{equation}\label{Eq:SeqPi}
\pi _1, \pi_2, \ldots , \pi _n    
\end{equation} 
such that every two consecutive faces in this sequence have a common $T$-edge. More precisely, for every $i=1, \ldots, n$, we have $$\partial \pi_i= e_{i-1}u_ie_i^{-1}v_i^{-1},$$ where $u_i$, $v_i$ are paths containing no $T$-edges,  and $e_{i-1}$, $e_i$ are $T$-edges. Thus, $e_i$ is the edge shared by $\pi_i$ and $\pi_{i+1}$ if $1\le i<n$ (see Fig. \ref{Fig1}). If, in addition, $e_0=e_n$, we call $\Sigma $ a \textit{$T$-annulus}. We call the paths $u=u_1\ldots u_n$ and $v=v_1\ldots v_n$ (respectively, $\Lab(u)$ and $\Lab(v)$) the \textit{sides} (respectively \textit{side labels}) of the $T$-band $\Sigma $. The sides and side labels of a $T$-annulus are defined up to simultaneous cyclic permutations caused by the choice of the first face in the sequence (\ref{Eq:SeqPi}).   
\end{defn}

\begin{figure}[H]
    \centering
    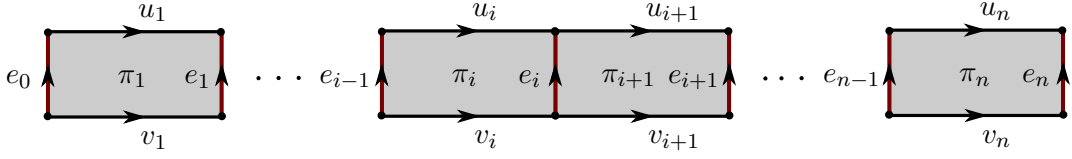
    \caption{A $T$-band formed by faces $\pi_1, \ldots, \pi_n$. $T$-edges are highlighted in red.}
    \label{Fig1}
\end{figure}

In this paper, $T$-bands will often arise as subdiagrams of a given van Kampen diagram $\Delta $ over (\ref{Eq:bi}). We say that a $T$-band in $\Delta$ is \textit{maximal} if it is not contained in any other $T$-band with a greater number of $T$-faces. Clearly, every maximal $T$-band $\Sigma $ of $\Delta$ is either a $T$-annulus or shares exactly $2$ edges ($e_0$ and $e_n$ in the notation of the previous paragraph) with $\partial \Delta$. In the latter case, we call $\Sigma $ a \textit{radial $T$-band} if $\Delta $ is annular and $e_0$ and $e_n$ belong to distinct components of $\partial \Delta$; if $\Delta $ is a disk diagram, or if $\Delta $ is annular and $e_0$ and $e_n$ belong to the same component of $\partial \Delta$, we say that $\Sigma $ is a \textit{$T$-arch}. 

For every $T$-annulus $\Sigma$ in $\Delta$, one of the sides of $\Sigma$ is contained in the region of the plane bounded by the other; we call these sides \textit{inner} and \textit{outer}, respectively. Further, we say that a $T$-annulus $\Sigma$ in $\Delta$  is \textit{$0$-homotopic} if it is contained in a disk subdiagram of $\Delta$. Thus, non-$0$-homotopic $T$-annuli cannot occur in disk diagrams; however, they can occur in annular diagrams when a $T$-band ``wraps around" the inner boundary component; we call such $T$-annuli \textit{concentric.} 

\begin{figure}[H]
    \centering
    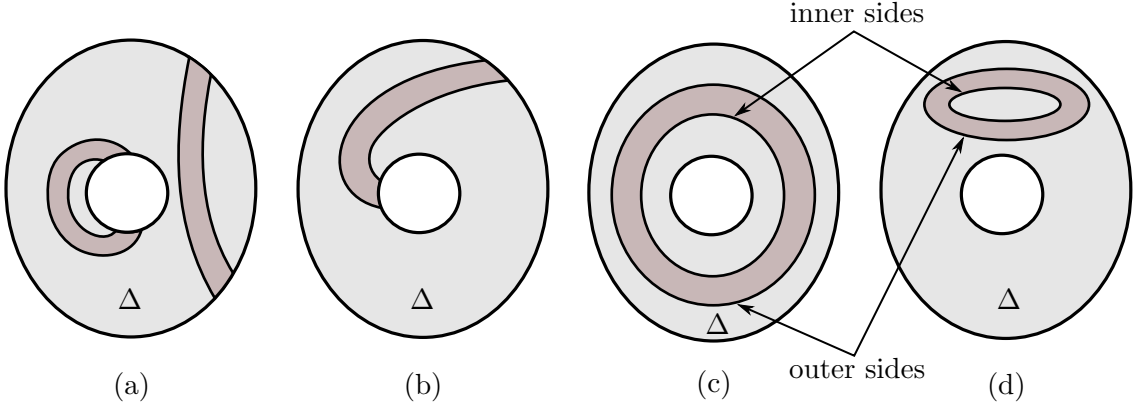
    \caption{Types of $T$-bands in van Kampen diagrams: (a) $T$-arches, (b) radial $T$-band, (c)~concentric $T$-annulus, (d) $0$-homotopic $T$-annulus.}
    \label{Fig2}
\end{figure}

We say that two $T$-bands in $\Delta$ \textit{intersect} if they share a $T$-face.
\begin{lem}\label{Lem:bands}
   In any van Kampen diagram over a $T$-bilateral presentation (\ref{Eq:bi}), distinct maximal $T$-bands do not intersect.
\end{lem}

\begin{proof}
If two distinct maximal $T$-bands intersect, the boundary of one of their shared $T$-face must have more than two $T$-edges, contradicting the definition of a $T$-bilateral presentation. 
\end{proof}

\paragraph{3.2. Embedding of the base group.}
Throughout the rest of this section, we fix a group $G$ given by a $T$-bilateral presentation (\ref{Eq:bi}). 

\begin{defn}[\bf Base group]
  Associated to $G$ is the \textit{base group} 
\begin{equation}\label{Eq:base}
G_0=\langle X\mid \mathcal R_0\rangle ,
\end{equation}
where $\mathcal R_0$ consists of all relators from $\mathcal R$ that do not contain any letters from $T^{\pm 1}$.  
\end{defn}

In general, the natural homomorphism $G_0\to G$ induced by the identity map $X\to X$ is not necessarily an embedding. Our next goal is to provide sufficient conditions for the map to be injective and for the image of $G_0$ to be undistorted in $G$. To this end, we first introduce the following.

\begin{defn}[\bf $T$-conjugate words]
    Let $U$ and $V$ be two words in the alphabet $X^{\pm 1}$. We say that $U$ and $V$ are \textit{$T$-conjugate} if there exists a $T$-band over (\ref{Eq:bi}) with side labels $U$ and $V$.
\end{defn}

\begin{rem}
    It is useful to keep in mind that $T$-conjugacy is generally not an equivalence relation on the set of words over $X^{\pm 1}$. 
\end{rem}

To guarantee the desired properties of the natural homomorphism $G_0 \to G$, we will consider bilateral presentations satisfying certain additional conditions. Although these conditions may appear difficult to verify, they will be immediate for all specific presentations considered in this paper. Recall that for a word $W$ over the alphabet $X^{\pm 1}$, $|W|_X$ denotes the length of the element of $G_0$ represented by $W$ with respect to the generating set $X$.

\begin{itemize}
    \item[($\dag$)]  \hypertarget{A}{\textit{For any $T$-conjugate words $U$, $V$ over $X^{\pm 1}$, we have $U=_{G_0}1$ if and only if $V=_{G_0}1$.}}
    \item[($\ddag$)]  \hypertarget{AA}{\textit{For any $T$-conjugate words $U$, $V$ over $X^{\pm 1}$, we have $|U|_X=|V|_X$.}}
\end{itemize}

Note that \hyperlink{AA}{($\ddag$)} is obviously stronger than \hyperlink{A}{($\dag$)} since a word $W$ over $X^{\pm 1}$ represents $1$ in $G_0$ if and only if $|W|_X=0$. 

\begin{ex}\label{Ex:HNN1}
    In the setting of Example \ref{Ex:HNN}, the presentation 
(\ref{Eq:HNN}) always satisfies \hyperlink{A}{($\dag$)}. Indeed, the elements of $G_0$ represented by the side labels of any $T$-band are $a$ and $\phi(a)$ for some $a\in A$; the property \hyperlink{A}{($\dag$)} follows since $\phi$ is an isomorphism. Similarly, this presentation satisfies \hyperlink{AA}{($\ddag$)} whenever $|a|_X=|\phi(a)|_X$ for all $a\in A$.  
\end{ex} 

\begin{defn}[\bf $T$-minimal van Kampen diagrams]\label{def: T minimal van Kampen diagrams}
    A van Kampen diagram $\Delta $ over~(\ref{Eq:bi}) is said to be \textit{$T$-minimal} if it contains the smallest possible number of $T$-faces among all diagrams of the same homotopy type (diskor annular) with the same labels of boundary components. 
\end{defn}

\begin{lem}\label{Lem:0-hom}
    Let (\ref{Eq:bi}) be a $T$-bilateral presentation satisfying \hyperlink{A}{($\dag$)} and let $\Delta$ be a van Kampen diagram over (\ref{Eq:bi}).  If $\Delta$ is $T$-minimal, then it contains no $0$-homotopic $T$-annuli. 
\end{lem}

\begin{proof}
Arguing by contradiction, we consider a $0$-homotopic $T$-annulus $\Sigma$, so that there are no other $T$-annuli inside the subdiagram $\Xi$ of $\Delta $ bounded by the inner side of $\Sigma$ (see Fig.~\ref{Fig3}).  Note that $\Xi$ is a disk diagram (even if $\Delta $ is not) since $\Sigma$ is $0$-homotopic. Let $u$ and $v$ be the inner and outer sides of $\Sigma$, respectively. Since $\Lab (u)$ contains no letters from $T^{\pm 1}$, $\Xi$ cannot contain $T$-bands. Therefore, $\Xi $ is a diagram over $(\ref{Eq:base})$, and we have $\Lab(u)=_{G_0}1$. By  \hyperlink{A}{($\dag$)}, we have $\Lab(v)=_{G_0}1$. Hence, there exists a disk diagram $\Xi^\prime$  over $(\ref{Eq:base})$ with boundary label $\Lab(v)$. We can then remove the subdiagram of $\Delta$ bounded by $v$ and patch the resulting hole with $\Xi'$, identifying $\partial \Xi'$ and $v$ in the obvious way. This operation does not affect the homotopy type and the boundary label of $\Delta$; however, the resulting diagram $\Delta'$ has fewer $T$-faces than $\Delta$, contrary to the assumption that $\Delta$ is $T$-minimal.
\end{proof}

\begin{figure}[H]
    \centering
    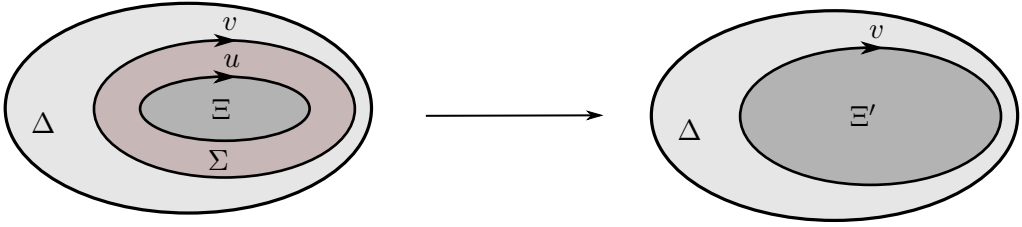
    \caption{Transforming the diagram $\Delta $ in the proof of Lemma \ref{Lem:0-hom}.}
    \label{Fig3}
\end{figure}

The first claim of the following proposition generalizes the well-known property of HNN-extensions to all bilateral presentations satisfying \hyperlink{A}{($\dag$)}.

\begin{prop}\label{Prop:embed}
Let $G$ be a group given by a $T$-bilateral presentation (\ref{Eq:bi}). 
\begin{enumerate}
    \item[(a)] If the presentation (\ref{Eq:bi}) satisfies  \hyperlink{A}{($\dag$)}, then the identity map $X\to X$ extends to an embedding of the base group $G_0\to G$. Thus, we can think of $G_0$ as a subgroup of $G$. 
    \item[(b)] If the presentation (\ref{Eq:bi}) satisfies  \hyperlink{AA}{($\ddag$)}, then the embedding $G_0\to G$ is isometric; that is, for every $g\in G_0$ we have $|g|_X=|g|_{X\cup T}$. 
\end{enumerate} 
\end{prop}

\begin{proof}
Clearly, $id\colon X\to X$ extends to a homomorphism $G_0\to G$. To prove (a), we need to verify that if  $W$ is a word over the alphabet $X^{\pm 1}$ and $W=_G1$, then $W=_{G_0}1$. To this end, let $\Delta $ be a $T$-minimal disk diagram over (\ref{Eq:bi}) such that $\Lab (\partial \Delta)\equiv W$. Since $W$ contains no letters from $T^{\pm 1}$, every $T$-band in $\Delta $ must be a $T$-annulus. Further, since $\Delta$ is a diskdiagram, every $T$-annulus in $\Delta $ must be $0$-homotopic. Using the minimality assumption and Lemma \ref{Lem:0-hom}, we conclude that $\Delta $ contains no $T$-bands at all. Thus, $\Delta $ is a diagram over the presentation (\ref{Eq:base}), and we have $W=_{G_0}1$. 

We now prove (b). Given $g\in G_0$, we consider a geodesic word $W$ over $(X\cup T)^{\pm 1}$ representing $g$ in $G$. We want to show that $W$ contains no letters from $T^{\pm 1}$. Arguing by contradiction, suppose it does. Let $Z$ be any word over $X^{\pm 1}$ representing $g$ in $G$ and let $\Delta $ be a $T$-minimal disk diagram over (\ref{Eq:bi}) with boundary label $ZW^{-1}$; thus, $\partial \Delta$ can be represented as a concatenation $zw^{-1}$, where $\Lab (z)\equiv Z$ and $\Lab(w)\equiv W$. 

Since $z$ contains no $T$-edges, every $T$-edge $e$ of $w$ must either be traveled by $w$ at least twice (in the opposite directions), or must be contained in a $T$-arch sharing another $T$-edge $f$ with $w$. In either case, we have a decomposition $$w=w_1ew_2f^{-1}w_3,$$ where $w_1$, $w_2$, $w_3$ are subpaths of $w$, and $e$ and $f$ are $T$-edges. 

 {\it Case 1.} Suppose first that $e=f$. Since $ew_2f^{-1}$ is a loop, we obtain $\Lab (ew_2f^{-1})=_{G}1$. Hence, $\Lab(w_1w_3)=_G\Lab(w)=_Gg$. Obviously, we have $\| w_1w_3\| < \| W\|= |g|_{X\cup T}$, which contradicts the assumption that $W$ is geodesic.
 
\begin{figure}[H]
    \centering
    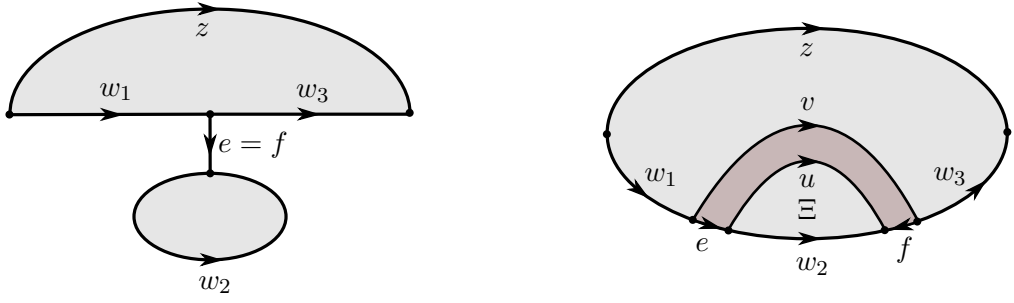
    \caption{Two cases in the proof of Proposition \ref{Prop:embed}.}
    \label{Fig6}
\end{figure}

{\it Case 2.} Assume now that there is a $T$-arch $\Sigma $ in $\Delta$ containing $e$ and $f$. Passing to another pair of $T$-edges if necessary, we can assume that $w_2$ contains no $T$-egdes.

Let $u$ and $v$ be sides of $\Sigma$. Up to adjusting the notation, we can assume that $v_-=e_-$, and $v_+=f_-$ (see Fig. \ref{Fig6}). Note that $$\Lab (w_1vw_3) =_G\Lab(w_1ew_2f^{-1}w_3) =_Gg.$$ 

By our assumption, $w_2$ contains no $T$-edges. Therefore, the subdiagram $\Xi$ bounded by $uw_2^{-1}$ contains no $T$-bands by Lemma \ref{Lem:0-hom} and, therefore, is a diagram over (\ref{Eq:base}). Thus, $\Lab (u)=_{G_0}\Lab (w_2)$. Combining with \hyperlink{AA}{($\ddag$)}, we obtain 
$$
|\Lab(v)|_{X\cup T}\le |\Lab(v)|_{X}= |\Lab(u)|_{X}\le \| \Lab(w_2)\|.
$$
Therefore, 
$$
\begin{array}{rcl}
|g|_{X\cup T} & \le & | \Lab (w_1)|_{X\cup T} + | \Lab (v)|_{X\cup T} + | \Lab (w_3)|_{X\cup T} = \\ && \\&& \|\Lab(w_1)\| + \| \Lab(w_2)\| + \| \Lab(w_3)\| =\| W\| -2.
\end{array}
$$
This contradicts the choice of $W$ again.

Thus, $W$ is a word over $X^{\pm 1}$, and we have $|g|_X\le \| W\| =|g|_{X\cup T}$. The opposite inequality is obvious.
\end{proof}

\paragraph{3.3. $T$-minimality and conjugacy in groups given by bilateral presentations.} We now obtain some useful results about the structure of annular van Kampen diagrams over a $T$-bilateral presentation (\ref{Eq:bi}). 

\begin{defn}[\bf $T$-minimal words] Let $W$ be a word  over $(X\cup T)^{\pm 1}$. We call the number of letters from $T^{\pm 1}$ in $W$ the $T$-content of $W$ and denote it by $\Cont_T(W)$. Further, we say that $W$ is \textit{$T$-minimal} if $\Cont_T(W)\le \Cont_T(U)$ for any word $U$ over $(X\cup T)^{\pm 1}$ such that $U=_GW$.  Finally, $W$ is \textit{cyclically $T$-minimal} if every cyclic shift of $W$ is $T$-minimal.
\end{defn}

\begin{ex}
    The standard presentation of the group $$\ZZ\ast \ZZ^2 =\langle a,s,t \mid [s,t]=1\rangle $$ is $\{t\}$-bilateral. It is easy to see from the normal form theorem for free products that the word $W= t^{-1}ats$ is $\{t\}$-minimal. However, it is not cyclically $\{t\}$-minimal; indeed, $atst^{-1}$ represents the same element of  $\ZZ\ast \ZZ^2$ as $as$ and $\Cont_{\{ t\}} (atst^{-1})=2 >0=\Cont_{\{ t\}}(as)$.
\end{ex}

In the proof of the next two lemmas, we will use the following notation and terminology. Let $\Delta $ be an annular van Kampen diagram over a $T$-bilateral presentation (\ref{Eq:bi}) and let $\Sigma$ be a $T$-arch with sides $u$ and $v$ sharing two $T$-edges $e$ and $f$ with $\partial_{ext}\Delta$. Adjusting the notation if necessary, we can assume that $\partial_{ext}\Delta = peqf^{-1}$, where $pv$ is a loop bounding a simply connected subdiagram $\Xi$ of $\Delta$ (see Fig. \ref{Fig4}). In this notation, we call $p$ (considered as a subset of $\partial \Delta$) the \textit{support} of $\Sigma$ and call $u$ (respectively, $v$) the \textit{external} (respectively, \textit{internal}) side of $\Sigma$. 

\begin{figure}[H]
    \centering
    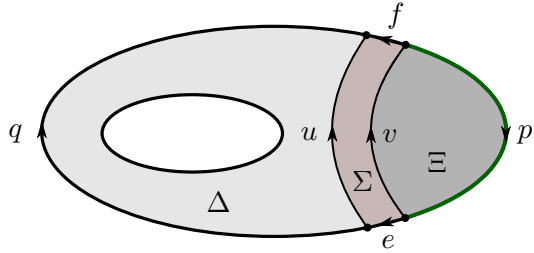
    \caption{The support of the $T$-arch $\Sigma$ is highlighted in green.}
    \label{Fig4}
\end{figure}

\begin{lem}\label{Lem:NoArches}
    Let $\Delta $ be a $T$-minimal annular diagram over a $T$-bilateral presentation (\ref{Eq:bi}) satisfying \hyperlink{AA}{($\ddag$)}. If the label of $\partial_{ext}\Delta$ (respectively,  $\partial_{int}\Delta$) is cyclically $T$-minimal, then $\Delta $ contains no $T$-arches sharing $T$-edges with $\partial_{ext}\Delta$ (respectively,  $\partial_{int}\Delta$). In particular, if both $\Lab(\partial_{ext}\Delta)$ and $\Lab(\partial_{int}\Delta)$ are cyclically $T$-minimal, then one of the following holds. 
\begin{enumerate}
    \item[(a)]  All $T$-bands in $\Delta $ are radial.
    \item[(b)] All $T$-bands in $\Delta $ are concentric $T$-annuli.
\end{enumerate}
\end{lem}
\begin{proof}
   Suppose first that $\Lab(\partial_{ext}\Delta)$ is cyclically $T$-minimal and $\Delta $ contains a $T$-arch $\Sigma$ sharing two $T$-edges with $\partial_{ext}\Delta $ (the proof for $\partial_{int}\Delta $ is similar). In the notation of Fig.~\ref{Fig4}, we have $\Lab(f^{-1}peu)=_G1$, and, therefore, $\Lab (f^{-1}peq)=_G\Lab(u^{-1}q)$. However,  $\Cont_T(\Lab(u^{-1}q))\le \Cont_T(\Lab(f^{-1}peq))-2$, which contradicts cyclic $T$-minimality of $\Lab(\partial_{ext}\Delta)$. Thus, $\Delta $ cannot contain $T$-arches sharing $T$-edges with $\partial_{ext}\Delta $. 

   If both $\Lab(\partial_{ext}\Delta)$ and $\Lab(\partial_{int}\Delta)$ are cyclically $T$-minimal, $\Delta $ cannot contain $T$-arches at all. By Lemma \ref{Lem:0-hom}, it also cannot contain $0$-homotopic $T$-annuli. Thus, every $T$-band in $\Delta $ is either radial or a concentric $T$-annulus. It remains to note that $\Delta $ cannot contain both by Lemma \ref{Lem:bands}.
\end{proof}

The usefulness of Lemma \ref{Lem:NoArches} stems from the following result, which reduces the study of conjugacy in groups given by bilateral presentations to considering elements represented by cyclically $T$-minimal words. Below, we denote by $g^G$ the conjugacy class of an element $g$ in $G$.

\begin{lem}\label{Lem:CTM}
    Let $G$ be a group given by a $T$-bilateral presentation (\ref{Eq:bi}) satisfying \hyperlink{AA}{($\ddag$)}. For any $g\in G$, there exists a cyclically $T$-minimal word $W$ over $(X\cup T)^{\pm 1}$ representing an element $g'\in g^G$ such that $\| W\| \le |g|_{X\cup T}$ and $c(g, g')\le |g|_{X\cup T}$.
\end{lem}

\begin{proof}
Fix any $g\in G$. Among all words over $(X\cup T)^{\pm 1}$ representing elements of $g^G$, we choose a word $V$ with minimal possible $\Cont_T(V)$. It is straightforward to verify that any such word $V$ is cyclically $T$-minimal. Further, we fix any geodesic word $U$ over $(X\cup T)^{\pm 1}$ representing $g$ in $G$ and let $\Delta $ be a $T$-minimal annular diagram over (\ref{Eq:bi}) such that $\partial _{ext}\Delta $ and $\partial _{int}\Delta $ are labeled by $U$ and $V$, respectively. 

Suppose first that $\Delta $ contains no $T$-arches. Since $U$ is geodesic, $\partial_{ext}\Delta$ is a simple loop. Therefore, every $T$-edge of $\partial_{ext}\Delta$ must either belong to $\partial_{int}\Delta$ or to a radial $T$-band; this implies the equality $\Cont_T(U)=\Cont_T(V)$. Hence, $U$ is cyclically $T$-minimal, and we can let $g'=g$ and $W\equiv U$.

Assume now that the set $\mathcal A$ of all $T$-arches in $\Delta$ is non-empty. Note that every $\Sigma \in \mathcal A$ must begin and end on $\partial_{ext}\Delta$ by Lemma~\ref{Lem:NoArches}. We define a partial order $\preceq$ on $\mathcal A$ by declaring $\Sigma \preceq \Sigma'$ if and only if the support of $\Sigma$ is contained in the support of $\Sigma'$. Let $\Sigma _1, \ldots, \Sigma_n$ be the list of all $\preceq$-maximal $T$-arches in $\Delta$ and let $u_1, \ldots , u_n$ be their external sides. Permuting the indices if necessary, we can assume that there exist subpaths $q_1, \ldots , q_n$ of $\partial \Delta$ such that $u_1q_1\ldots u_nq_n$ is a loop bounding an annular subdiagram $\Xi$ of $\Delta $ (see Fig.~\ref{Fig5}). 

\begin{figure}[H]
    \centering
    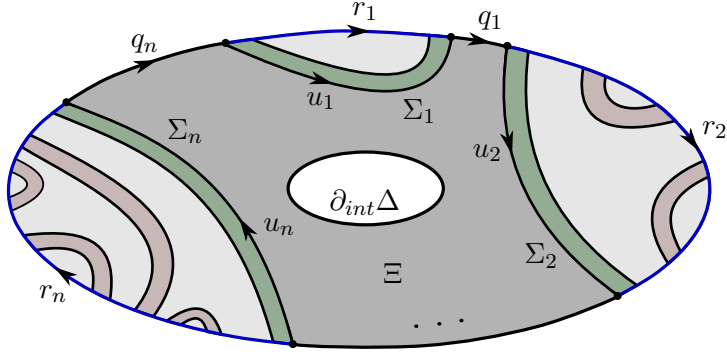
    \caption{$\preceq $-maximal $T$-arches of $\Delta$ are highlighted in green; the subpaths $r_1, \ldots, r_n$ are highlighted in blue.}
    \label{Fig5}
\end{figure}

Note that $\Xi$ cannot contain $T$-arches; indeed, by Lemma \ref{Lem:NoArches}, no $T$-arch of $\Xi$ can begin and end on $\partial_{int} \Xi = \partial_{int} \Delta$. Note also that the paths $u_1, \ldots, u_n$ contain no $T$-edges. Therefore, the $T$-edges shared by any $T$-arch of $\Xi$ and $\partial \Xi$ must belong to $q_i$ and $q_j$ for some $i,j\in \{ 1, \ldots, n\}$. Clearly, the existence of such a $T$-arch contradicts the choice of $\Sigma _1, \ldots, \Sigma_n$. As above, we conclude that every $T$-edge of $\partial_{ext}\Xi$ must either belong to $\partial_{int}\Delta$ or to a radial $T$-band in $\Xi$. Therefore, 
\begin{equation}\label{Eq:Cont}
\sum\limits_{i=1}^n\Cont_T(\Lab (q_i))=\Cont_T(V).    
\end{equation}

The label of $\partial \Xi=u_1q_1\ldots u_nq_n$ can be very long as there is no upper bound on the lengths of the paths $u_1, \ldots, u_n$. To obtain the desired word $W$, we modify $\Lab (u_1q_1\ldots u_nq_n)$ as follows. First, let $r_1, \ldots, r_n$ be the subpaths of $\partial \Delta$ such that $$\partial_{ext}\Delta = r_1q_1\ldots r_nq_n.$$ Since $\Lab(r_i)=_G\Lab(u_i)$ for all $i$, we have $|\Lab (u_i)|_{X\cup T}\le \ell (r_i)$. By part (b) of Proposition~\ref{Prop:embed}, there are words $U_i$ in $X^{\pm 1}$ representing the same elements of $G_0$ as $\Lab(u_i)$ such that $\| U_i\|=|\Lab (u_i)|_{X\cup T}\le \ell (r_i)$. Let 
$$
W\equiv U_1\Lab(q_1)\ldots U_n\Lab(q_n) 
$$
and let $g'$ be the element of $G$ represented by $W$. Clearly, $g'=_G\Lab(u_1q_1\ldots u_nq_n)$. Since $\Lab(u_1q_1\ldots u_nq_n)$ is equal to a cyclic shift of $U$ in $G$, we have $c(g,g')\le \| U\|=|g|_{X\cup T}$. Further, since all $U_i$ are words over $X^{\pm 1}$, we have $\Cont_T(W)=\Cont_T(V)$ by (\ref{Eq:Cont}). Hence, $W$ is cyclically $T$-minimal. It remains to note that
$$
\| W\| \le \sum\limits_{i=1}^n (\| U_i\| + \| \Lab(q_i)\|) \le \sum\limits_{i=1}^n (\ell (r_i) + \ell (q_i))= \ell (\partial_{ext}\Delta)= \| U\| = |g|_{X\cup T}. 
$$
\end{proof}

\section{Proof of Theorem \ref{Thm:main1}} \label{sec:proof of main1}

In this section, we construct a group with a prescribed conjugator length function, thereby proving the second direction of Theorem \ref{Thm:main1}.

Let $f\colon \NN\to \NN$ be an increasing function. Let $$\mathcal A=\{ x,y,z,r,s\}$$ and let 
$$
U_i=x^irx^i, \;\;\; V_i=y^isy^i\;\;\;, W_i=z^{f(i)}tz^{f(i)}
$$
for all $i\in \NN$. Let 

\begin{equation}\label{Eq:PresG main1}
    G=\Big\langle \mathcal A, t\left | \; W_i^{-1}U_iW_i=V_i\; (i\in \NN) \Big\rangle \right. .
\end{equation}

For each $i\in \NN$, consider the words: $$C_i=z^{-f(i)}x^irx^iz^{f(i)},\ D_i=z^{f(i)}y^isy^iz^{-f(i)},$$ 
and denote by $c_i,\ d_i$ the elements of $F(\calA)$ represented by $C_i,\ D_i$, respectively. Define $$C= \langle c_1,c_2,\dots\rangle\text{ and } D= \langle d_1,d_2,\dots \rangle .$$
\begin{lem}\label{lem:C free}
The subgroup $C$ is freely generated by $\{c_i : i\in\NN\}$, and the subgroup $D$ is freely generated by $\{d_i : i\in\NN\}$.
\end{lem}
\begin{proof}
We prove the claim for $C$; the argument for $D$ is identical. Consider a nonempty reduced word
\begin{equation}\label{eq:c in C}
W=C_{i_1}^{\epsilon_1}C_{i_2}^{\epsilon_2}
\cdots C_{i_n}^{\epsilon_n}
\end{equation}
in the alphabet $\{C_i^{\pm1}:i\in\NN\}$, where
$\epsilon_k\in\{\pm1\}$ and $(i_{k+1},\epsilon_{k+1}) \neq(i_k,-\epsilon_k)$
for every $1\leq k<n$.

Expanding the factors gives
\begin{equation}\label{eq:W}
z^{-f(i_1)}x^{\epsilon_1i_1}r^{\epsilon_1}
\prod_{k=1}^{n-1}
\left(
x^{\epsilon_ki_k}
z^{f(i_k)-f(i_{k+1})}
x^{\epsilon_{k+1}i_{k+1}}
r^{\epsilon_{k+1}}
\right)
x^{\epsilon_ni_n}z^{f(i_n)}.
\end{equation}
For each $1\leq k<n$, the subword
\[
x^{\epsilon_ki_k}
z^{f(i_k)-f(i_{k+1})}
x^{\epsilon_{k+1}i_{k+1}}
\]
represents a non-trivial element of $F(x,z)$. Indeed, if
$i_k\neq i_{k+1}$, then, since $f$ is strictly increasing,
$f(i_k)\neq f(i_{k+1})$, so the word contains a nonzero power of $z$.
If $i_k=i_{k+1}$, then the word reduces to $x^{(\epsilon_k+\epsilon_{k+1})i_k}$; since $\epsilon_{k}=\epsilon_{k+1}$, this gives $x^{2\epsilon_k i_k}$.
Thus, after freely reducing the $F(x,z)$-subwords, \eqref{eq:W} is in
reduced normal form in $F(x,z)*\langle r\rangle=F(x,z,r)$.
In particular, it represents a non-trivial element. Hence
$\{c_i:i\in\NN\}$ freely generates $C$.
\end{proof}

\bigskip 

Returning to the study of $G$, note that the $i$-th relation from (\ref{Eq:PresG main1}) can be re-arranged as $t^{-1} C_i t = D_i$. Indeed:

    \begin{align*}
        \underbrace{z^{-f(i)} t^{-1} z^{-f(i)}}_{W_i^{-1}} \underbrace{x^irx^i}_{U_i} \underbrace{z^{f(i)} t z^{f(i)}}_{W_i} = \underbrace{y^isy^i}_{V_i} \\  
        t^{-1}\underbrace{z^{-f(i)}x^irx^iz^{f(i)}}_{C_i}t = \underbrace{z^{f(i)}y^isy^iz^{-f(i)}}_{D_i}
    \end{align*}
It follows that the presentation (\ref{Eq:PresG main1}) can be written as 
\begin{equation}\label{Eq:PresG main1 HNN}
G = %F(\mathcal{A}) *_\phi =
\Big\langle \mathcal A, t \left | \; t^{-1}C_it=D_i\; (i\in \NN) \Big\rangle \right. .
\end{equation}

Note that each relator from $\mathcal{R}=\{t^{-1}C_itD_i^{-1}:i\in \NN\}$ contains exactly two letters from $\{t\}^{\pm 1}$, and therefore we get:
\begin{prop}
    The presentation (\ref{Eq:PresG main1 HNN}) for $G$ is $\{t\}$-bilateral. Its associated base group is the group $G_0=F(\calA)$, the free group generated by $\calA$.
\end{prop}

Our next goal is to show that presentation (\ref{Eq:PresG main1 HNN}) satisfies \hyperlink{A}{($\dag$)} and \hyperlink{A}{($\ddag$)}.

Let $\mathcal{A}^*$ denote the set of all (not necessarily reduced) words over $\mathcal{A}^{\pm 1}$. Define a map $\phi : \mathcal{A}^* \to \mathcal{A}^*$ by interchanging $x \leftrightarrow y$, $r \leftrightarrow s$, and $z \leftrightarrow z^{-1}$, and extending it letter-by-letter to arbitrary words. The map $\phi$ is an isometry with respect to the natural word metric on $\mathcal{A}^*$, and we have that $\phi(C_i)=D_i$ for all $i\in \NN$. 

Since $\phi(w)^{-1}=\phi(w^{-1})$ for any letter $w\in \calA^{\pm 1}$, the map $\phi$ induces a well-defined automorphism, still denoted by $\phi$, of the free group $F(\mathcal{A})$, which is also an isometry with respect to the natural word metric $|\cdot|_{\mathcal{A}}$ on $F(\calA)$.

%Note that the subgroup $C$ (respectively, $D$) is generated freely by $\{c_i: i\in \NN\}$ (respectively, $\{d_i: i\in \NN\}$), and, moreover, that $C\cap D=\{1\}$.
The automorphism $\phi$ of $F(\calA)$ restricts to an isometric isomorphism $\phi:C \to D$, which satisfies $\phi(c_i) = d_i$ for all $i \in \mathbb{N}$. It follows, in particular, that $G$ is the HNN extension of $F(\calA)$ relative to $\phi$, see Example \ref{Ex:HNN}.

Observe that the side labels of each $\{t\}$-face in a van Kampen diagram $\Delta$ over (\ref{Eq:PresG main1 HNN}) are $C_i$ and $\phi(C_i)$ for some $i\in \NN$. %; we call $i$ \emph{the rank of $\pi$}. 
It follows that the side labels of each $\{t\}$-band in $\Delta$ are $W$ and $\phi(W)$ for some word $W$ representing an element in the subgroup $C$. That $\phi$ is an isometry with respect to $|\cdot|_{\calA}$ implies that $|W|_\calA=|\phi(W)|_\calA$. We obtain:

\begin{lem}\label{lem:PresG satis ddag}
    The presentation (\ref{Eq:PresG main1 HNN}) satisfies \hyperlink{A}{($\dag$)} and \hyperlink{A}{($\ddag$)}.
\end{lem}

By Proposition \ref{Prop:embed}, property \hyperlink{A}{($\ddag$)} implies that the natural homomorphism $\iota \colon G_0 \to G$ is an isometric embedding. We may therefore identify $G_0 = F(\mathcal{A})$ with its image in $G$ and regard it as a subgroup of $G$. In particular, we will view the subgroups $C$ and $D$ of $F(\mathcal{A})$, as well as their elements, as subgroups and elements of $G$. 

For the remainder of this section, we study the group $G$ defined in (\ref{Eq:PresG main1 HNN}). In particular, the notation $|\cdot|$ denotes the word metric on $G$ with respect to the generating set $\calA\cup \{t\}$.

We now study van Kampen diagrams over the presentation (\ref{Eq:PresG main1 HNN}) for $G$. Recall (Convention \ref{con:orientation}) that unless indicated otherwise, the labels of boundaries of faces and boundary components of diagrams are read clockwise.

Since $T=\{t\}$ is a singleton, we shall omit braces and write $t$-edge, $t$-face, $t$-band, $t$-annulus, and $t$-arch instead of $\{t\}$-edge, $\{t\}$-face, etc. The same convention applies to $t$-minimal diagrams, $t$-conjugate words, and $t$-minimal words.

\begin{defn}
Let $\pi$ be a $t$-face in a van Kampen diagram over (\ref{Eq:PresG main1 HNN}). By the form of the defining relators, there is a unique $i\in \NN$ such that the side labels of $\pi$ are $C_i$ and $D_i$, or $C_i^{-1}$ and $D_i^{-1}$. We call $i$ \emph{the rank of $\pi$}.
\end{defn}

Let $\gamma \in \mathcal{A}$. An edge labeled by either $\gamma$ or $\gamma^{-1}$ will be called a \emph{$\gamma$-edge}. 
The next lemma provides useful observations about $t$-faces sharing an $r$- or an $s$-edge. 
\begin{lem}\label{lem:shared r edges}
Let $\Delta$ be a $t$-minimal diagram over (\ref{Eq:PresG main1 HNN}). Let $\pi, \pi'$ be a pair of $t$-faces in $\Delta$ sharing an $r$-edge or an $s$-edge. Then, $\pi$ and $\pi'$ have different ranks.
% and moreover, the maximal common sub-path of $\partial\pi_i$ and $\partial \pi_j$ has length $2\min\{i,j\}+1$.
\end{lem}
\begin{proof}
    Assume $\pi$ and $\pi'$ share an $r$-edge, to be denoted by $e_r$. The case where they share an $s$-edge is completely analogous, so we omit its proof. 
    Reading the boundaries of $\pi,\pi'$ clockwise, we can write $\partial\pi=e_r q$ and $\partial\pi'=e_r^{-1}l$. Let $\Xi$ be the disk subdiagram bounded by the loop $ql$.
   
    If $\pi$ and $\pi'$ have the same rank, then $\Lab(q)\equiv \Lab(l^{-1})$. It follows that we can remove $\Xi$ from $\Delta$ and patch the resulting hole by identifying the subpaths $q$ and $l^{-1}$ in the obvious way. This operation does not affect the homotopy type and the boundary label of $\Delta$; however, it reduces the number of $t$-faces by $2$. This contradicts the $t$-minimality assumption on $\Delta$, completing the proof.
\end{proof}

Let $\Sigma$ be a $t$-band in a van Kampen diagram over (\ref{Eq:PresG main1 HNN}). Recall that the labels of its two sides are: $$C_{i_1}^{\pm1}\dots C_{i_n}^{\pm1} ~~~\text{and}~~~ \phi(C_{i_1}^{\pm1}\dots C_{i_n}^{\pm1})=D_{i_1}^{\pm1}\dots D_{i_n}^{\pm1}$$ for some $i_1,\dots,i_n\in \NN$. Consequently, one side of $\Sigma$ consists entirely of $x$-, $z$-, and $r$-edges, while the other consists entirely of $y$-, $z$-, and $s$-edges. We call the former the \emph{$C$-side}
of $\Sigma$ and the latter the \emph{$D$-side}.

We have:
\begin{lem}\label{lem: r edge cases}
Let $\Sigma$ be a radial $t$-band in a $t$-minimal diagram $\Delta$ over (\ref{Eq:PresG main1 HNN}). Then any $r$-edge in $\Sigma$ either belongs to $\partial\Delta$, or is shared by a maximal $t$-band different than $\Sigma$.
\end{lem}
\begin{proof}
Suppose not. Then there exist faces $\pi,\pi'$ in $\Sigma$ and an $r$-edge $e_r$, such that $e_r$ occurs on $\partial\pi$, and $e_r^{-1}$ on $\partial\pi'$. By the form of defining relators, $\Lab(\partial \pi)$ cannot contain both $r$ and $r^{-1}$; consequently, $\pi\neq \pi'$. 
It follows that the $C$-side of $\Sigma$ contains a subpath of the form $e_r u e_r^{-1}$. Up to maybe replacing $u$ by a subpath of $u$, we may assume $u$ contains no $r$-edges. Therefore, $\pi$ and $\pi'$ are consecutive $t$-faces in $\Sigma$, that is, they share a $t$-edge, $e_t$.

Reading both boundaries clockwise, we can write $\partial\pi=e_r q_1 e_t q_2$ and $\partial\pi'=l_2 e_t^{-1}l_1 e_r^{-1}$. Since $\partial\pi$ contains exactly two $t$-edges, we have that one of $q_1,q_2$ does not contain a $t$-edge, call it $q_i$, $i\in\{1,2\}$. Since $\Sigma$ is radial, the subdiagram $\Xi$ bounded by $q_{i}l_i$ is a disk subdiagram. Since $q_il_i$ contains at most one $t$-edge (namely, the one possibly contributed by $l_i$), $\Xi$ does not contain any $t$-bands. Indeed, a maximal $t$-band in $\Xi$ must be a $0$-homotopic $t$-annulus, which is impossible, by the $t$-minimality assumption and Lemma \ref{Lem:0-hom}.
In particular, we have that  $\Lab(q_il_i)=_{F(\calA)}1$, from which it follows that $\pi$ and $\pi'$ have the same rank, contradicting Lemma \ref{lem:shared r edges}.
\end{proof}

Recall that $C$ is freely generated by $\{C_i:i\in \NN\}$. Consider a reduced word in the alphabet
$\{C_i^{\pm1}:i\in\mathbb{N}\}$:
\begin{equation}\label{eq:W over C}
W=C_{i_1}^{\epsilon_1} C_{i_2}^{\epsilon_2} \cdots C_{i_n}^{\epsilon_n},
\end{equation}
where $\epsilon_k\in\{\pm1\}$ and 
$(i_{k+1},\epsilon_{k+1})
\neq (i_k,-\epsilon_k)$
for every $1\leq k<n$. Denote by $A_i=x^irx^i$.

\begin{claim}\label{claim:A blocks survive}
Let $W$ be as in (\ref{eq:W over C}). Expanding the $C$-syllables in $W$, freely reducing over $F(x,z,r)$, no letter belonging to any block $A_i^{\epsilon_i}=(x^irx^i)^{\epsilon_i}$ is canceled.

Moreover, if $W$ is cyclically reduced as a word in $\{C_i^{\pm1}:i\in \NN\}$, then the same holds also if cyclically reducing over $F(x,z,r)$.
\end{claim}
\begin{proof}
Two consecutive blocks occur in the form $$A_{i_k}^{\epsilon_k}
z^{f(i_k)-f(i_{k+1})}
A_{i_{k+1}}^{\epsilon_{k+1}}.$$
If $f(i_k)\neq f(i_{k+1})$, the middle power of $z$ is non-trivial, so the two blocks are separated. Otherwise $f(i_k)= f(i_{k+1})$. Since $f$ is injective, $i_k=i_{k+1}:=j$, which implies, since (\ref{eq:W over C}) is reduced, that $\epsilon_k=\epsilon_{k+1}=:\epsilon_j$. In this case, $A_{i_k}^{\epsilon_k}
z^{f(i_k)-f(i_{k+1})}
A_{i_{k+1}}^{\epsilon_{k+1}}=A^{\epsilon_j}_{j}A^{\epsilon_j}_{j}$. 

\textcolor{teal}
Assume moreover that $W$ is cyclically reduced. The only additional cancellation that may occur when cyclically reducing is between the last and first $C$-syllables. The corresponding $A$-blocks occur cyclically in the 
form $A_{i_n}^{\epsilon_n} z^{f(i_n)-f(i_1)}A_{i_1}^{\epsilon_1}$.
If the middle power of $z$ is non-trivial, the two $A$-blocks are separated. Otherwise, injectivity of $f$ implies $i_n=i_1$. Since $W$ is cyclically reduced as a word in $\{C_i^{\pm1}:i\in\NN\}$ we must then have $\epsilon_n=\epsilon_1$. Thus cyclic reduction introduces no cancellation in any $A$-block, and the moreover statement follows.
\end{proof}

We will repeatedly use the following observation. Suppose $U\equiv U_1rU_2$ is a word over $\calA^{\pm1}$, and the letter $r$ survives the free reduction of $U$, then $|U|_{F(\calA)}=|U_1|_{F(\calA)}+1+|U_2|_{F(\calA)}$. Since the embedding $F(\calA)=G_0\to G$ is isometric, the same equality holds when the norm is taken in $G$: $|U|=|U_1|+1+|U_2|$.

As a preliminary step toward bounding $\CL_G$, we study the word-metric distance between the boundary components of annular diagrams over (\ref{Eq:PresG main1 HNN}). A key ingredient is the following technical lemma, which relies on the structure of words in $C$.

\begin{lem}\label{lem:r-edge split}
    Suppose a word $W$ over $\calA^{\pm 1}$ represents an element in $C$, and let $l\in \NN$. Then: $$|Wz^{-f(l)}x^{\pm l}|\geq f(l)+l.$$
\end{lem}
\begin{proof}
If $W$ represents the identity element in $G$, then $$|Wz^{-f(l)}x^{\pm l}|=|z^{-f(l)}x^{\pm l}|=f(l)+l,$$ which satisfies the desired inequality.  
Therefore, we may assume $W$ is non-trivial in $G$.

By Lemma \ref{lem:C free},  $$W=_GC_{i_1}^{\epsilon_{1}}C_{i_2}^{\epsilon_{2}}\cdots C_{i_n}^{\epsilon_{n}}$$ for some $n\geq 1$, where $\epsilon_{k}\in \{\pm 1\}$, and such $(i_{k+1},\epsilon_{k+1})\neq (i_k,-\epsilon_k$) for every $1\leq k<n$. Denote $$U\equiv Wz^{-f(l)}x^{\pm l}.$$
We now prove that $|U|\geq f(l)+l$ by induction on $n$.

Let $n=1$. Then $W=_G C_{i}^{\epsilon}$ for some $i\in \NN,\epsilon\in\{1,-1\}$, and so 
$$U=_G z^{-f(i)}x^{\epsilon i}r^{\epsilon}x^{\epsilon i}z^{f(i)-f(l)}x^{\pm l}.$$ If $i\neq l$, then $f(i)\neq f(l)$ and so the right hand side is freely reduced; in particular, $|U|\geq f(l)+l$. If $i=l$, then $U=_{G}z^{-f(i)}x^{\epsilon i}r^{\epsilon}x^{\epsilon i\pm l}$, where the right hand side is freely reduced; it follows that $|U|\geq f(i)+i=f(l)+l$.

    Let $n\geq 2$ and suppose the statement holds for $n-1$. In particular, denoting $W'=C_{i_1}^{\epsilon_{1}}C_{i_2}^{\epsilon_{2}}\cdots C_{i_{n-1}}^{\epsilon_{n-1}}$ and $U_1=W'z^{-f(i_n)}x^{\epsilon_n i_n}$ we have that 
    \begin{equation}\label{eq:|U1|}
        |U_1|\geq f(i_n)+i_n.
    \end{equation}
    It follows that \begin{eqnarray*} U& =_G & W'C_{i_n}^{\epsilon_{n}}z^{-f(l)}x^{\pm l}\\
    & \equiv &\underbrace{W'z^{-f(i_n)}x^{\epsilon_n i_n}}_{U_1}r^{\epsilon_n}\underbrace{x^{\epsilon_n i_n}z^{f(i_n)}z^{-f(l)}x^{\pm l}}_{U_2}.
    \end{eqnarray*} 
%    Note that the letter $r^{\epsilon_n}$ in the above expression do not cancel out. Indeed, $U_2$ contains no $r^{\pm 1}$, and as for $U_1$: cancellation cannot happen by the assumption that $(i_{k+1},\epsilon_{k+1})\neq (i_k,-\epsilon_k)$.
By Claim \ref{claim:A blocks survive}, $r^{\epsilon_n}$ in the displayed equation does not cancel out, thus:
%It follows that:
\begin{eqnarray*}
        |U| & =& |U_1|+1+|U_2| \\
        & \underset{(\ref{eq:|U1|})}{\geq} & f(i_n)+i_n+1+|U_2|.
    \end{eqnarray*}
    If $i_n\geq l$ this is at least $f(l)+l$, completing the proof. If $i_n<l$, then  $|U_2|= i_n+|f(i_n)-f(l)|+l=i_n-f(i_n)+f(l)+l$, in which case $|U|\geq f(l)+l$ as required.
\end{proof}

The next lemma and claim are direct corollaries of Lemma \ref{lem:r-edge split}.

\begin{lem}\label{lem:ell(p)>2f(i)+2i+1}
Let $u$ be a side of a $t$-band $\Sigma$ in a $t$-minimal diagram over (\ref{Eq:PresG main1 HNN}). If $\Sigma$ contains a $t$-face of rank $l$, where $l\in \NN$, then:
 $$|\Lab(u)|\ge 2f(l)+2l+1.$$ 
\end{lem}
\begin{proof} 
Since $\phi$ is an isometry, both side labels of $\Sigma$ have same norm. It therefore suffices to prove the lemma under the assumption that $u$ is the $C$-side of $\Sigma$.
We have that  
$$\Lab(u)=C_{i_1}^{\epsilon_1}C_{i_2}^{\epsilon_2}\cdots C_{i_n}^{\epsilon_n},$$ where $C_{i_k}^{\epsilon_k}$ is the label of the $C$-side of the $k^{\text{th}}$ $t$-face in $\Sigma$.
Since $\Delta$ is $t$-minimal, $(i_{k+1},\epsilon_{k+1})\neq (i_k,-\epsilon_k)$. Indeed, suppose $(i_{k+1},\epsilon_{k+1})=(i_k,-\epsilon_k)$, then the union of the two corresponding $t$-faces forms a disk subdiagram whose boundary label is trivial in $F(x,z,r)$. As explained in previous proofs, this contradicts $t$-minimality.

Since $\Sigma$ contains a $t$-face of rank $l$, there is $1\leq j\leq n$ such that $i_j=l$. Up to possibly reversing the orientation of $u$, we may further assume that $\epsilon_{j}=1$. Denoting $W_1=C_{i_1}^{\epsilon_1}\cdots C_{i_{j-1}}^{\epsilon_{j-1}}$ and $W_2=C_{i_{j+1}}^{\epsilon_{j+1}}\cdots C_{i_{n}}^{\epsilon_{n}}$, we have:
\begin{equation}\label{eq:Lab(u)_lem4.8}
    \Lab(u)=_GW_1\underbrace{z^{-f(l)}x^{l}rx^{l}z^{f(l)}}_{C_{l}}W_2.
\end{equation} %where $W_1$ and $W_2$ are words representing elements of $C$. Note that the right hand side in (\ref{eq:Lab(u)_lem4.8}) may not be freely reduced. However, since $\Delta$ is $t$-minimal, we may assume that the only cancellations in (\ref{eq:Lab(u)_lem4.8}) happen along $z$-edges;\gil{needs to be explained more carefully} in particular, the letter $r$ in $C_l$ does not cancel out.
By Claim \ref{claim:A blocks survive}, the block $x^{l}rx^{l}$ in (\ref{eq:Lab(u)_lem4.8}) does not cancel out, and so:
$$|\Lab(u)| =  |W_1z^{-f(l)}x^{ l}|+|r|+|x^{l}z^{f(l)}W_2|.$$
By Lemma \ref{lem:r-edge split}, each of $|W_1z^{-f(l)}x^{ l}|$ and $|x^{l}z^{f(l)}W_2|$ is at least $f(l)+l$, completing the proof.
\end{proof}

\begin{claim}\label{claim:conj to C_l}
Let $W$ be a word in $C$ that is conjugate to $C_l$ in $F(\calA)$, then:
$$|W|\geq 2f(l)+2l+1.$$
\end{claim}
\begin{proof}
We first show that $W$ is conjugate to $C_l$ in $C$.

Write $W$ as a cyclically reduced word in the free basis $\{C_i:i\in \NN\}$, up to conjugacy in $C$. Namely: \begin{equation}\label{eq:W conj to C_l} W=_{F(\calA)}U^{-1}\underbrace{C_{i_1}^{\epsilon_{i_1}}C_{i_2}^{\epsilon_{i_2}}\cdots C_{i_n}^{\epsilon_{i_n}}}_{V}U,
\end{equation}  where $V$ is cyclically reduced as a word over $\{C_i:i\in \NN\}$. Expanding $C_{i_1},\dots,C_{i_n}$ and reducing over $F(\calA)$ the blocks $(x^{i_k}rx^{i_k})^{\epsilon_{k}}$ survive entirely, see Claim \ref{claim:A blocks survive}. Consequently, each $C_{i_k}$ contributes $r^{\pm 1}$ that survives the free reduction. Moreover, these surviving $r$-letters also survive cyclic reduction: cancellation of the first and last $r$-letters would force $(i_n,\epsilon_n)=(i_1,-\epsilon_1)$, contrary to cyclic reduction of $V$. It follows that the cyclically reduced form of $V$ in $F(\calA)$ contains exactly $n$ occurrences of $r^{\pm 1}$. On the other hand, $C_l$ is conjugate in $F(\calA)$ to $rx^{2l}$, which is cyclically reduced and contains exactly one $r$. Since conjugate elements of a free group have cyclically reduced representatives that differ only by cyclic permutation, this implies $n=1$. It follows that $V=C_{j}^{\epsilon}$ for some $j\in \NN$, $\epsilon\in \{\pm 1\}$. Since $C_j^{\epsilon}$ is conjugate in $F(\calA)$ to $r^\epsilon x^{\epsilon \cdot 2j}$, which is cyclically reduced, the same argument implies that $\epsilon=1$ and $j=l$. This shows that $W$ is conjugate to $C_l$ in $C$.

Re-write (\ref{eq:W conj to C_l}) as 
$$W=_{F(x,z,r)}U^{-1}\underbrace{z^{-f(l)}x^lrx^lz^{f(l)}}_{C_l}U.$$
Up to possibly changing $U$, we may assume that the right hand side is freely reduced as a word in $\{C_i:i\in \NN\}^{\pm 1}$. By Claim \ref{claim:A blocks survive}, the $r$-letter in $C_l$ does not cancel out, and so $$|W|=|U^{-1}z^{-f(l)}x^l|+|r|+|x^lz^{f(l)}U|.$$
By Lemma \ref{lem:r-edge split}, each of $|U^{-1}z^{-f(l)}x^l|$ and $|x^lz^{f(l)}U|$ is at least $f(l)+l$, completing the proof.
\end{proof}

\begin{lem}\label{lem:a single concentric}
A $t$-minimal annular diagram $\Delta$ over (\ref{Eq:PresG main1 HNN}) contains at most one concentric $t$-band.
\end{lem}

\begin{proof}
Suppose toward contradiction that $\Delta$ contains at least two concentric $t$-bands. Let $\Sigma_1$ and $\Sigma_2$ be distinct concentric $t$-bands in $\Delta$.  For $i=1,2$ denote by $u_i$ and $v_i$ the sides of $\Sigma_i$. Up to adjusting the notations, we may assume that $\Sigma_2$ is contained in the region of the
plane bounded by $\Sigma_1$, that $u_1$ is the inner side of $\Sigma_1$ and that $u_2$ is the outer side of $\Sigma_2$, see Fig. \ref{fig:concentric_bands}. 
\begin{figure}
 \centering%% Creator: Inkscape inkscape 0.92.3, www.inkscape.org
%% Accompanies image file 'fig7.pdf' (pdf, eps, ps)
\begingroup%
  \makeatletter%
  \providecommand\color[2][]{%
    \errmessage{(Inkscape) Color is used for the text in Inkscape, but the package 'color.sty' is not loaded}%
    \renewcommand\color[2][]{}%
  }%
  \providecommand\transparent[1]{%
    \errmessage{(Inkscape) Transparency is used (non-zero) for the text in Inkscape, but the package 'transparent.sty' is not loaded}%
    \renewcommand\transparent[1]{}%
  }%
  \providecommand\rotatebox[2]{#2}%
  \newcommand*\fsize{\dimexpr\f@size pt\relax}%
  \newcommand*\lineheight[1]{\fontsize{\fsize}{#1\fsize}\selectfont}%
  \ifx\svgwidth\undefined%
    \setlength{\unitlength}{236.84609793bp}%
    \ifx\svgscale\undefined%
      \relax%
    \else%
      \setlength{\unitlength}{\unitlength * \real{\svgscale}}%
    \fi%
  \else%
    \setlength{\unitlength}{\svgwidth}%
  \fi%
  \global\let\svgwidth\undefined%
  \global\let\svgscale\undefined%
  \makeatother%
  \begin{picture}(1,0.48825642)%
    \lineheight{1}%
    \setlength\tabcolsep{0pt}%
    \put(0,0){\includegraphics[width=\unitlength,page=1]{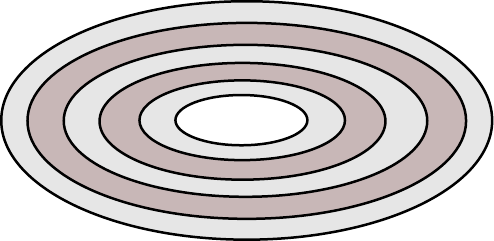}}%
    \put(0.06957707,0.2159184){\color[rgb]{0,0,0}\makebox(0,0)[lt]{\lineheight{1.25}\smash{\begin{tabular}[t]{l}$\Sigma_1$\end{tabular}}}}%
    \put(0.79962712,0.23384857){\color[rgb]{0,0,0}\makebox(0,0)[lt]{\lineheight{1.25}\smash{\begin{tabular}[t]{l}$\Xi$\end{tabular}}}}%
    \put(0.70666476,0.23416617){\color[rgb]{0,0,0}\makebox(0,0)[lt]{\lineheight{1.25}\smash{\begin{tabular}[t]{l}$\Sigma_2$\end{tabular}}}}%
    \put(0,0){\includegraphics[width=\unitlength,page=2]{fig7.pdf}}%
    \put(0.15512651,0.26305205){\color[rgb]{0,0,0}\makebox(0,0)[lt]{\lineheight{1.25}\smash{\begin{tabular}[t]{l}$u_1$\end{tabular}}}}%
    \put(0,0){\includegraphics[width=\unitlength,page=3]{fig7.pdf}}%
    \put(0.72824067,0.32004456){\color[rgb]{0,0,0}\makebox(0,0)[lt]{\lineheight{1.25}\smash{\begin{tabular}[t]{l}$u_2$\end{tabular}}}}%
  \end{picture}%
\endgroup%
\\
  \caption{Multiple concentric $t$-bands}  \label{fig:concentric_bands}
\end{figure}
Let $\Xi$ denote the annular diagram bounded between $\Sigma_1$ and $\Sigma_2$. Up to maybe replacing $\Sigma_1$ by a different concentric $t$-band, we may assume that $\Xi$ contains no concentric $t$-bands. Furthermore, $\Xi$ contains no $t$-arches or radial $t$-bands, as it is bounded away from $\partial\Delta$, and no $0$-homotopic annuli, by Lemma \ref{Lem:0-hom} and since $\Delta$ is $t$-minimal.
It follows that $\Xi$ contains no $t$-bands, namely, it is a van Kampen diagram over $F(\calA)$. 
By van Kampen Lemma \ref{Lem:van Kampen}(b), it follows that $\Lab(u_1)$ and $\Lab(u_2)$ are conjugate in $F(\calA)$. 

The cyclic version of Claim \ref{claim:conj to C_l} shows that a nontrivial $C$-word has a cyclically reduced representative containing both $x$ and $r$, whereas
every element of $D$ lies in $F(y, z, s).$ It follows that no non-trivial element of $C$ is conjugate in $F(\calA)$ to an element of $D$.
Therefore, either both $\Lab(u_1)$ and $\Lab(u_2)$ belong to $C$, or both belong to $D$. Without loss of generality, assume the former; the latter case is entirely analogous, with $\phi^{-1}$ replacing $\phi$.
It follows that $\Lab(v_1)=\phi(\Lab(u_1))$ and $\Lab(v_2)=\phi(\Lab(u_2))$. 

The map $\phi$, as defined on words over $\calA^{\pm 1}$, commutes with conjugation in $F(\calA)$. Using that $\Lab(u_1)$ and $\Lab(u_2)$ are conjugate in $F(\calA)$, it follows that also  $\phi(\Lab(u_1))$ and $\phi(\Lab(u_2))$ are conjugate in $F(\calA)$.
Consider the annular subdiagram of $\Delta$ which is the union of $\Sigma_1$, $\Xi$, and $\Sigma_2$. As the labels of its boundary components, $\phi(\Lab(u_1))$ and $\phi(\Lab(u_2))$, are conjugate in $F(\calA)$, we can remove it from $\Delta$ and patch the resulting hole by identifying the boundary components in the obvious way. This operation preserves the homotopy type and boundary labels of $\Delta$, but reduces the number of $t$-faces by at least $2$, contrary to the assumption that $\Delta$ is $t$-minimal.

We conclude that $\Delta$ contains at most one concentric $t$-band.
\end{proof}

Let $\Delta$ be an annular diagram. We denote the total length of its boundary components by: $$\ell(\partial\Delta)=\ell(\partial_{ext} \Delta)+\ell(\partial_{int} \Delta).$$

The next two lemmas show that, in $t$-minimal annular diagrams over (\ref{Eq:PresG main1 HNN}), the word-metric distance between the components of $\partial\Delta$ is bounded by a function of $\ell(\partial\Delta)$.

\begin{lem}
\label{lem: linear m connecting the boundary}
    Let $\Delta$ be $t$-minimal annular diagram over (\ref{Eq:PresG main1 HNN}) containing a radial $t$-band, whose boundary labels are cyclically $t$-minimal. 
    Then there exists a path $m$ connecting the two boundary components of $\Delta$, such that
    $$|\Lab(m)|\le \ell(\partial \Delta)+1$$
\end{lem}

\begin{proof}
Let $\Sigma$ be a radial $t$-band in $\Delta$. By Lemma \ref{lem: r edge cases}, each $r$-edge in the sides of $\Sigma$ is either shared with $\partial\Delta$, or with another maximal $t$-band, $\Sigma'$. We therefore consider two cases:

\emph{Case 1. There is an $r$-edge $e_r$ shared by $\Sigma$ and another maximal $t$-band, $\Sigma'$.}

By Lemma \ref{Lem:NoArches}, $\Sigma'$, like $\Sigma$, is radial. Let $u$ and $u'$ denote the $C$-sides of $\Sigma$ and $\Sigma'$, respectively. Then $u$ and $u'$ decompose as $$u=u_1e_r u_2 ~~\text{ and }~~ u'=u_2'e_r^{-1}u_1'.$$ 
Let $i$ (respectively, $j$) denote the rank of the $t$-face $\pi$ ($\pi'$) in $\Sigma$ ($\Sigma'$) whose boundary contains $e_r$ ($e_r^{-1}$). By Lemma \ref{lem:shared r edges}, $i\neq j$.

\begin{figure}
 \centering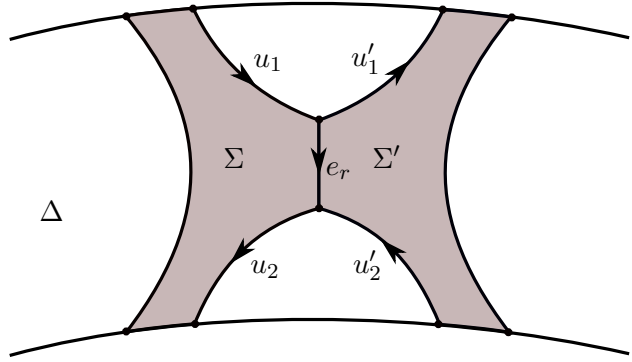\\
  \caption{Case 1 in the proof of Lemma \ref{lem: linear m connecting the boundary}}
\end{figure}

We claim that $$|\Lab(u_1)|\leq |\Lab(u_1u_1')|.$$ Indeed, the freely reduced form of $\Lab(u_1)$ ends either with $z^{\pm1 }x^{\pm i}$, or with $r^{\pm1}x^{\pm 2i}$; the latter occurs precisely when $\pi$ is preceded in $\Sigma$ by another $t$-face of rank $i$. 
Similarly, the freely reduced form of $\Lab(u_1')$ begins either with $x^{\pm j}z^{\pm1 }$, or with $x^{\pm 2j}r^{\pm1}$, where the latter occurs when $\pi'$ is followed in $\Sigma'$ by another $t$-face of rank $j$. It follows that at most $2\min\{i,j\}$ letters from each of $\Lab(u_1)$ and $\Lab(u_1')$ cancel when freely reducing the concatenation $\Lab(u_1)\Lab(u_1')$. 
On the other hand, after this cancellation, at least
$j+f(j)\geq 2j\geq 2\min\{i,j\}$ letters of $\Lab(u_1')$ remain. Hence, the letters contributed by $\Lab(u_1')$ that survive the cancellation compensate for all the letters canceled from $\Lab(u_1)$, giving
$|\Lab(u_1u_1')|\geq |\Lab(u_1)|$ as claimed.

Similarly, one shows that $$|\Lab(u_2)|\leq |\Lab(u_2'u_2)|.$$

Let $w_1$ be the subpath of $\partial\Delta$ with same starting point and ending point as the concatenation $u_1u_1'$. Consider the subdiagram bounded by the loop $u_1u_1'w_1^{-1}$. It is simply connected and therefore $\Lab(u_1u_1'w_1^{-1})=_G1$. 
It follows that $$|\Lab(w_1)|= |\Lab(u_1u_1')|.$$
Similarly, denoting by $w_2$ the subpath of $\partial \Delta$ with same starting and ending point as $u_2'u_2$, and considering the subdiagram bounded by the loop $u_2'u_2w_2^{-1}$, one gets that
$$|\Lab(w_2)|=|\Lab(u_2'u_2)|.$$
Recall that $u=u_1e_r u_2$. By Claim \ref{claim:A blocks survive}, the $r$-letter $\Lab(e_r)$ in $\Lab(u)$ does not cancel out, and so $|\Lab(u)| = |\Lab(u_1)|+1+|\Lab(u_2)|$. Using the four displayed relations, we obtain:
\begin{eqnarray*}
|\Lab(u)|-1 & = & |\Lab(u_1)| +|\Lab(u_2)| \\
    & \leq & |\Lab(u_1u_1')| + |\Lab(u_2'u_2)|\\
    & = & |\Lab(w_1)|+|\Lab(w_2)|\\
    & \leq & \ell(\partial\Delta). \\
\end{eqnarray*}
Taking $m=u$ completes the proof in this case.

\emph{Case 2. All $r$-edges in $\Sigma$ are shared with $\partial\Delta$.}

Denote by $\pi_1,\dots,\pi_n$ the $t$-cells that form $\Sigma$, such that for $i=1,\dots,n-1$, $\pi_i$ and $\pi_{i+1}$ share a $t$-edge and such that $\pi_1$ shares a $t$-edge with $\partial_{ext}\Delta$ and $\pi_n$ with $\partial_{int}\Delta$. Without loss of generality assume that $u$ initiates in $\partial_{ext}\Delta$ and terminates in $\partial_{int}\Delta$. 

Suppose, first, that the $r$-edge in $\pi_1$ occurs in $\partial_{int}\Delta$. Denote by $e_r$ this $r$-edge, and by $l$ the rank of $\pi_1$. The path $u$ decomposes as 
$$u=u_1e_ru_2,$$
and we have that $\Lab(u_1)=z^{-f(l)}x^{\pm l}$ has norm $|\Lab(u_1)|=f(l)+l$. Further, $\Lab(u_2)=x^{\pm l}z^{f(l)}W$ for a word $W$ representing an element in $C$. By Lemma \ref{lem:r-edge split}, it follows that $$|\Lab(u_2)|\geq f(l)+l.$$ 
Let $w$ be the subpath of $\partial_{int}\Delta$ having the same starting and ending point as $u_2$. Using the same argument as in case 1, consider the subdiagram bounded by the loop $u_2w^{-1}$. It is simply connected, and therefore $\Lab(u_2)=_G\Lab(w)$. It follows that \begin{eqnarray*}
    \ell(\partial\Delta)& \geq & |\Lab(w)| \\
    & = &|\Lab(u_2)|\\ 
%    & \geq & |\Lab(u)|-\Lab(u_1)\\
%    & \geq & (2f(l_1)+2l_1+1)-(f(l_1)+l_1)\\
    & \geq & f(l) + l\\
    & = & |\Lab(u_1)|.
\end{eqnarray*}
Note further that $u_1$ initiates from $\partial_{ext}\Delta$ and terminates at $\partial_{int}\Delta$; taking $m=u_1$ completes the proof under the assumption that the $r$-edge in $\pi_1$ occurs in $\partial_{int}\Delta$.

The same exact argument gives a proof in case the $r$-edge in $\pi_n$ occurs in $\partial_{ext}\Delta$.

We can therefore assume that neither of these happen, namely, that the $r$-edge from $\pi_1$ belongs to $\partial_{ext}\Delta$ and the $r$-edge from $\pi_n$ to $\partial_{int}\Delta$. Under this assumption, there must be $i\in \{1,\dots,n-1\}$, for which the $r$-edge of $\pi_{i}$ belongs to $\partial_{ext}\Delta$ but the $r$-edge of $\pi_{i+1}$ to $\partial_{int}\Delta$. Denote those edges by $e_r$ and $e_r'$ respectively. Then $u$ decomposes as 
$$u=u_1e_rm e_r' u_2,$$
where $m$ initiates from $\partial_{ext}\Delta$ and terminates at $\partial_{int}\Delta$. Denote further by $l$ the rank of $\pi_i$ and by $l'$ the rank of $\pi_{i+1}$. Then $\Lab(m)\equiv x^{\pm l}z^{f(l)}z^{-f(l')}x^{\pm l'}$ and therefore $$|\Lab(m)|\leq f(l)+l+f(l')+l'.$$ 
Let $w_1$ be the subpath of $\partial_{ext}\Delta$ with the same starting and ending points as $u_1$. Since the subdiagram bounded by the loop $u_1w_1^{-1}$ is simply connected, $\Lab(u_1)=_G\Lab(w_1)$ and so they have the same norm. Similarly, the label of the subpath $w_2$ of $\partial_{int}\Delta$ having the same starting and ending as $u_2$ has the same norm as $\Lab(u_2)$. Using Lemma \ref{lem:r-edge split}, we obtain that $|\Lab(u_1)|\geq f(l)+l$ and $|\Lab(u_2)|\geq f(l')+l'$. It follows that:
\begin{eqnarray*}
    \ell(\partial\Delta) & \geq & |\Lab(w_1)|+|\Lab(w_2)| \\
    & = & |\Lab(u_1)|+|\Lab(u_2)|\\
    & \geq & f(l)+l+f(l')+l'\\
    & \geq &|\Lab(m)|.
\end{eqnarray*}
This completes the proof of this case.
\end{proof}

\begin{lem}\label{lem:m in concentric case}
    Let $\Delta$ be a $t$-minimal annular diagram over (\ref{Eq:PresG main1 HNN}) consisting of a single concentric $t$-band. Denote $\ell(\partial\Delta)=\ell$.  %whose boundary labels are cyclically $t$-minimal. 
    Then there exists a path $m$ connecting the two boundary components of $\Delta$, such that
    $$|\Lab(m)|\le 2f(\ell)+1$$
\end{lem}
\begin{proof}
Let $\Sigma$ be the concentric $t$-band in $\Delta$. Let $\pi$ be a $t$-face in $\Sigma$, and denote its rank by $i$; so $\Lab(\partial\pi)\equiv t^{-1}C_i^{\pm 1}tD_i^{\mp 1}$. 
We decompose $\partial\pi=m'q_1mq_2$ such that:
$$\Lab(\partial\pi)=\underbrace{z^{-f(i)}t^{-1}z^{-f(i)}}_{\Lab(m')}\underbrace{x^{\pm i}r^{\pm1}x^{\pm i}}_{\Lab(q_1)}\underbrace{z^{f(i)}tz^{f(i)}}_{\Lab(m)} \underbrace{y^{\pm i}s^{\pm1}y^{\pm i}}_{\Lab(q_2)}.$$

Note that each boundary label of $\Delta$ is equal, in $F(\calA)$, to a cyclic conjugate of the corresponding side label of $\Sigma$. From here it is straightforward to show that $i\leq |\Lab(\partial\Delta)|\leq \ell(\partial\Delta)\leq \ell$.
In particular, $|\Lab(m)|=2f(i)+1\leq 2f(\ell)+1$.

We now claim that all $x$-edges in $\partial\pi$ lie on $\partial\Delta$. Suppose not; then there exists an $x$-edge $e$ occurring in $\partial\pi$, such that $e^{-1}$ belongs to $\partial\pi'$, for some face $\pi'$ in $\Sigma$. By the form of the defining relators, the boundary label of a single face cannot contain both $x$ and $x^{-1}$; consequently, $\pi\neq \pi'$. It follows that the $C$-side of $\Sigma$ contains a subpath of the form $e u e^{-1}$. Up to maybe replacing $u$ by a subpath of $u$, we may assume $u$ contains no $x$-edges, and so the $t$-faces $\pi$ and $\pi'$ are consecutive in $\Sigma$, that is, they share a $t$-edge. 
Since $\Lab(eue^{-1})$ is a subword of a word in the generators $\{C_i:i\in\NN\}$, while $u$ contains no $x$-edges but is both preceded and followed by an $x$-edge, it follows that $$\Lab(u)=z^{f(i)}z^{-f(i')},$$ where $i$ and $i'$ denote the ranks of $\pi$ and $\pi'$, respectively. The subdiagram bounded by the loop $u$ is simply connected, and therefore $\Lab(u)=_G1$, from which it follows that $f(i)=f(i')$. Since $f$ is strictly increasing, we get $i=i'$, which, together with the fact that $\pi$ and $\pi'$ share an $x$-edge, implies that $\Lab(\partial\pi')=\Lab(\partial\pi)^{-1}$. Consequently, the boundary label of $\pi\cup \pi'$ is trivial in $F(\calA)$. Therefore, the disk subdiagram $\pi\cup \pi'$ can be removed from $\Delta$ and the remaining hole can be patched in the obvious way, without affecting the homotopy type and boundary label of $\Delta$. However, the resulting diagram would have two $t$-faces less than the number of $t$-faces in $\Delta$, which contradicts the assumption that $\Delta$ is $t$-minimal.
We conclude that all $x$-edges in $\partial\pi$ lie on $\partial\Delta$.

Similarly, one shows that all $y$-edges in $\partial\pi$ lie on $\partial\Delta$.

It remains to observe that, since $\Sigma$ is concentric and $q_1$ and $q_2$ intersect opposite sides of $\Sigma$, those $x$-edges and the $y$-edges belong to different boundary components of $\Delta$. Consequently, $m$ starts on one boundary component of $\Delta$ and ends on the other.
%Since $\Sigma$ is concentric and $q_1$ and $q_2$ belong to different sides of $\Sigma$, they belong to different boundary components of $\Delta$. For $j=1,2$, let $\partial_j\Delta$ denote the boundary component of $\Delta$ containing $q_j$.
%It follows that $m$ initiates at $\partial_1\Delta$ and terminates at $\partial_2\Delta$. It remains to bound $|\Lab(m)|$.

%For $j=1,2$, let $w_j$ be the subpath of $\partial_i\Delta$ with same starting and ending point as $q_j$. The loop $q_jw_j^{-1}$ bounds simply connected subdiagram, and so, by van Kampen Lemma \ref{Lem:van Kampen}, $w_j=_G q_j$. It follows that, for $j=1,2$, $|\Lab(w_j)|=|\Lab(q_j)|=2i+i$.  Therefore, $$ \ell(\partial\Delta)\geq \ell(w_1)+\ell(w_2)\geq |\Lab(w_1)|+|\Lab(w_2)|\geq 4i+2.$$ 
\end{proof}

The proof of the second direction of Theorem \ref{Thm:main1} for non-decreasing, at least linear functions is split into the following two lemmas. 

Consider a non-decreasing, at least linear function $f\colon \NN\to \NN\cup \{0\}$. Then $\hat f(n)=f(n)+n$ is strictly increasing, $\hat f(n)\geq n$, and $\hat f\sim f$. Indeed, let $C,d$ be such that $n\leq Cf(Dn)$, then $$f(n)\leq \hat f(n)\leq f(n) + Cf(Dn)\leq (C+1)f(\max\{D,1\}n).$$ Hence, up to maybe replacing $f$ by $\hat f$, we may assume that $f$ is strictly increasing.

\begin{lem}[Upper bound]\label{lem:upper bound}
    $\CL_G\preceq f$.
\end{lem}

\begin{proof}
    Let $a,b\in G$ be two conjugate elements, of total norm $|a|+|b|\leq n$ for some $n\in \NN$, we show that $c(a,b)\le 5f(n)$.
    
    By Lemma \ref{Lem:CTM}, there exist cyclically $t$-minimal words $W$ and $Z$ over $(\calA\cup \{t\})^{\pm1}$ representing elements $a'\in a^G$ and $b'\in b^G$, respectively, such that $\|W\|\leq |a|$, $c(a,a')\le |a|$, $\|Z\|\leq |b|$, and
    $c(b,b')\le |b|$. In particular, we have that:
    \begin{equation}\label{eq:a+b<=n}
    \|W\|+\|Z\|\leq n ~~\text{and}~~    
    c(a,a')+c(b,b')\leq n.    
    \end{equation}
%We have that $c(a,b)\le c(a,a')+c(a',b')+c(b,b')\le c(a',b')+|a|+|b|\le c(a',b')+n$. It is therefore enough to prove $c(a',b')\le 4f(n)$. 

Let $\Delta$ be a $t$-minimal annular diagram over (\ref{Eq:PresG main1 HNN}), whose boundary labels are $W$ and $Z$; such diagram exists by van Kampen Lemma \ref{Lem:van Kampen}(b).
If $\Delta$ contains no faces, then $a',b'$ are conjugate in the free group $F(\mathcal{A}\cup \{t\})$, which implies $c(a',b')\le n$. Using (\ref{eq:a+b<=n}), it follows that $c(a,b)\leq c(a,a')+c(a',b')+c(b,b')\leq 2n$, which proves the assertion in this case. 
Therefore, we may assume $\Delta$ has faces. 

\emph{Claim: There is a path $m$ connecting the two components of $\partial\Delta$, for which $|\Lab(m)|\le 3f(n)$.}

If $\Delta$ contains a radial $t$-band, Lemma \ref{lem: linear m connecting the boundary} provides a path $m$, connecting the two components of $\partial\Delta$, such that $|\Lab(m)|\leq \ell(\partial\Delta)+1$. Since 
$\ell(\partial\Delta)=\|W\|+\|Z\|\le n$, and since $f$ is strictly increasing and takes only integer values, $n\leq f(n)$, and so the assertion of the claim holds in this case.

Otherwise, since (\ref{Eq:PresG main1 HNN}) satisfies \hyperlink{A}{($\ddag$)} (Lemma \ref{lem:PresG satis ddag}) and the boundary labels of $\Delta$ are cyclically $t$-minimal, Lemma \ref{Lem:NoArches} implies that all $t$-bands in $\Delta$ are concentric. 
By Lemma \ref{lem:a single concentric}, $\Delta$ cannot contain multiple concentric $t$-bands, and we deduce that $\Delta$ consists of a single concentric $t$-band.
It then follows from Lemma \ref{lem:m in concentric case} that there exists a path $m$, connecting the two components of $\partial\Delta$, such that $|\Lab(m)|\leq 2f(\ell(\partial\Delta))+1\leq 2f(n)+1$.
This completes the proof of the claim.

Let $m$ be a path as in the claim and denote $V=\Lab(m)$, so $|V|\leq 3f(n)$. Since $m$ connects the two components of $\partial\Delta$, the words $Z$ and $W$ admit cyclic shifts $Z'$ and $W'$ respectively, such that, up to maybe replacing $V$ by $V^{-1}$:  $$V^{-1}Z'V=_GW'.$$
%It follows that $$c(Z',W')\leq |V|\le 3f(n).$$
By equation (\ref{eq:a+b<=n}), and since $c(W,W')\leq \|W\|$ and $c(Z,Z')\leq \|Z\|$, we get: 
\begin{eqnarray*}
    c(W,Z) & \leq & c(W,W')+c(W',Z')+c(Z,Z') \\
    &\leq & \|W\| + |V| + \|Z\| \\
    &\leq & 3f(n)+n,
\end{eqnarray*}
and therefore, $c(a',b')\leq 3f(n)+n$.
Using (\ref{eq:a+b<=n}) once again, we get: \begin{eqnarray*}
    c(a,b) & \leq & c(a,a')+c(a',b')+c(b,b') \\
    &\leq & |a|+ c(a',b')+|b| \\
    &\leq &3f(n)+2n.
\end{eqnarray*}
Since $n\leq f(n)$, this implies the assertion, completing the proof of the lemma.
\end{proof}

\begin{lem}[Lower bound]\label{Lem:LB}
    $\CL_G\succeq f$. 
\end{lem}

\begin{proof}
Recall that for $i\in \mathbb{N}$ we denote the words $U_i=x^irx^i$ and $V_i=y^isy^i$, and that for all $i$, the elements represented by $U_i$ and $V_i$ are conjugate in $G$, see (\ref{Eq:PresG main1}).
Fix $n\in \mathbb{N}$ and let $W$ be a word of length $c(U_n,V_n)$ for which $WU_n W^{-1} =_G V_n$. We show that $\|W\|\geq f(n)$.

Recall that a \textit{pinch} is a word of the form $t^{-1}At$ or $tBt^{-1}$ for some words $A\in C$, $B\in D$. Without loss of generality, can assume that $W$ contains no pinches. Indeed, every pinch represents an element of $F(\mathcal A)$ and, therefore, can be replaced with a word in $F(\mathcal A)$ without increasing the length of $W$ by Proposition \ref{Prop:embed} (b).  

Let $\Delta$ be a disk, $t$-minimal van Kampen diagram over (\ref{Eq:PresG main1 HNN}) with boundary label $WU_n W^{-1} V_n^{-1}$. Decompose the boundary of $\Delta$ accordingly, as $\partial\Delta =wp(w')^{-1}q^{-1}$.

Since $U_n$ and $V_n$ are not conjugate in $G_0=F(\calA)$, $W$ must contain a letter $t$ or $t^{-1}$. We consider the $t$-arch $\Sigma$ of $\Delta$ starting at the last $t$-edge of $w$. Denote by $e$ and $e'$ the $t$-edges of $\Sigma$, and by $u$ and $v$ its sides, so that $\partial \Sigma = eue'^{-1}v^{-1}$. Clearly, $e$ and $e'$ cannot both belong to $w$ as otherwise $W$ would contain a pinch. Since $p$ and $q$ contain no $t$-edges, we conclude that $e'$ belongs to $w'$.

Let $w_0$ be the suffix of $w$ starting at $e_+$, and $w_0'$ the suffix of $w'$ starting at $e'_+$. Let $u$ denote the side of $\Sigma$ starting at $e_+$ and ending at $e'_+$, see Fig. \ref{fig:lower bound}.
\begin{figure}\label{fig:lower bound}
 \centering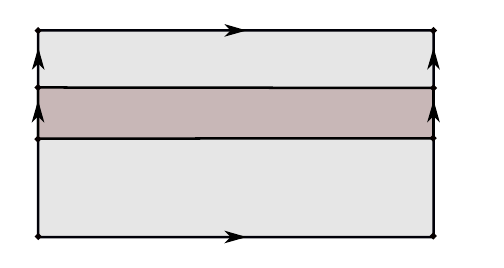\\
  \caption{Proof of Lemma \ref{Lem:LB}}
\end{figure}

By a symmetric argument, the $t$-band $\Sigma'$ containing the last $t$-edge of $w'$ must share a $t$-edge with $w$. Since $t$-bands cannot intersect, we conclude that $\Sigma=\Sigma'$ and $e'$ is the last $t$-edge of $w'$. Since $\Lab(w)\equiv\Lab(w')$, it follows that $\Lab(w_0)\equiv\Lab(w_0')$. By minimality of $\Delta$, the subdiagram bounded by $w_0pw_0^{-1}u^{-1}$ contains no $t$-faces; therefore, we have $$\Lab(u)=_{F(\mathcal A)}\Lab(w_0)\Lab(p)\Lab(w_0)^{-1}.$$ Since  $w_0$ contains no $t$-edges, $\Lab(u)$ is conjugate to $\Lab(p)$ in $F(\calA)$; since $\Lab(p)=x^nrx^n$ is conjugate to $C_n$ in $F(\calA)$, so is $\Lab(u)$. 
By Claim \ref{claim:conj to C_l}, $|\Lab(u)|\geq 2f(n)+2n+1$. 
Therefore,  
$$    
2\ell(w_0)=\ell(w_0)+\ell(w_0') = \ell\Big(w_0p(w_0')^{-1}\Big)-\ell(p) \ge  |\Lab(u)|-(2n+1)\ge  2f(n).
$$
It follows that $\|W\|=\ell(w)\ge \ell(w_0) \geq f(n)$.

By our assumption, $\|W\|=c(U_n,V_n)$ and $\|U_n\|+\|V_n\|=4n+2$. Thus, we obtain $\CL_G(6n)\ge \CL_G(4n+2)\geq f(n)$, which gives $\CL_G \succeq f$, as required.
\end{proof}

\section{$\CL$ and finite extensions} \label{sec:proof of main2}

 This section is devoted to the proof of Theorem~\ref{Thm:main2}. Let $G$ be the group constructed in Section~\ref{sec:proof of main1}, with respect to a prescribed function. We construct an index-two extension of $G$ and prove that it has linear conjugator length function.

 Let $f\colon \NN\to \NN$ be an increasing function. Let $\mathcal{A}$ and $G$ be as defined in Section~\ref{sec:proof of main1}. 
  % \begin{rem}
  %   Since functions are considered up to asymptotic equivalence $\sim$, we may replace $f$ by $\widehat{f}(n) = f(n) + n + 1.$ Hence, up to asymptotic equivalence, we may assume that $f(n) > n$ and that $f$ is strictly increasing. 
  %   \end{rem}
    Observe that the map $\alpha \colon (\calA\cup{t})^{\pm1} \to (\calA\cup{t})^{\pm1}$ defined by
    $\alpha(x)=y,\;\;\alpha(y)=x,\;\; \alpha (r)=s,\;\; \alpha (s)=r,\;\; \alpha(z)=z^{-1},\;\; \alpha(t)=t^{-1}
    $
    extends to an order $2$ automorphism of $G$. 
    
    Inspired by the Collins-Miller construction \cite{CM77}, we let 
    
    $$
    E=G\rtimes_\alpha \langle a \mid a^2=1\rangle.
    $$ 
    
    We use the standard exponential notation $x^y:=y^{-1}xy$.
    
    \begin{lem}\label{Lem:Epres}
       Denote $B=\{r,x,z\}$. The group $E$ admits the presentation: \begin{equation}\label{eq:PresE}    
       \left\langle
             %a,\, b,\, r,\, x,\, z\,
             B\cup \{a,\,b\}
        \left|\,
               a^2=1,\, b^2=1,\, z^a=z^{-1},\, 
             \left[b, (x^irx^i)^{z^{f(i)}}\right]=1\,
         \right.\right\rangle.
         \end{equation}
    \end{lem}
    
    \begin{proof}
    The proof uses Tietze transformations. By definition, the group $E$ is generated by the elements $x,y,z,r,s,t,a$ subject to the following relations:
    \begin{equation}\label{Eq:rel1}
    t^{-1}\underbrace{z^{-f(i)}x^irx^iz^{f(i)}}_{C_i}t = \underbrace{z^{f(i)}y^isy^iz^{-f(i)}},
    \end{equation}
    (a cyclic permutation of $w_i^{-1}u_iw_i=v_i$) and
    \begin{equation}\label{Eq:rel3}
    a^2=1,\;\; x^a=y, \;\; y^a=x, \;\; r^a=s,\;\; s^a=r,\;\; z^a=z^{-1}, \;\; t^a=t^{-1}.
    \end{equation}
    Using (\ref{Eq:rel3}), we can  eliminate the redundant generators $y$ and $s$ and rewrite (\ref{Eq:rel1}) in the form
    $$
    \left(z^{-f(i)}x^irx^iz^{f(i)}\right)^t = \left( z^{-f(i)}x^irx^iz^{f(i)}\right)^ a.
    $$
    Conjugating both sides by $a=a^{-1}$ and introducing a new generator $b=ta$, we obtain the relation $[b, z^{-f(i)}x^irx^iz^{f(i)}]=1$. It remains to omit the redundant generator $t=ba$ and rewrite the last relation in (\ref{Eq:rel3}) in the form $b^2=1$.
    \end{proof}

Since every relator in~\eqref{eq:PresE} contains exactly two occurrences of letters from $\{a,b\}^{\pm1}$, the presentation~\eqref{eq:PresE} is an $\{a,b\}$--bilateral presentation for $E$. Hence the band machinery of Section~\ref{sec:bilateral pres} applies to van Kampen diagrams over~\eqref{eq:PresE}. We record the conventions specific to this presentation. An $a$--face is either an $a$--bigon, corresponding to the relation $a^2=1$, or a non-bigon face corresponding to the relation $z^a=z^{-1}$. In the latter case, the two sides of the face are the opposite $z$--edges, with labels $z$ and $z^{-1}$.

Similarly, a $b$--face is either a $b$--bigon, corresponding to the relation $b^2=1$, or a non-bigon face corresponding to one of the relations involving
\[
C_i\equiv z^{-f(i)}x^irx^iz^{f(i)}.
\]
For such a non-bigon $b$--face,  $i \in \mathbb{N}$ is called its \emph{rank}, and we often denote the face by $\pi_i$. As in Section~\ref{sec:proof of main1}, let $c_i$ be the element of $F(B)$ represented by $C_i$, and set $C=\langle c_1,c_2,\dots\rangle.$ Note that $C$ is freely generated by $\{c_i\mid i\in\mathbb N\}$, see Lemma~\ref{lem:C free}.

\begin{lem}\label{lem: E ddag}
The presentation~\eqref{eq:PresE} satisfies \hyperlink{A}{($\dag$)} and \hyperlink{A}{($\ddag$)} as an $\{a,b\}$--bilateral presentation.
\end{lem}

\begin{proof}
It is enough to verify \hyperlink{A}{($\ddag$)}, since \hyperlink{A}{($\ddag$)} immediately implies \hyperlink{A}{($\dag$)}.

Let $U$ and $V$ be $a$--conjugate words over $B^{\pm1}$. Then there is an $a$--band $\Sigma$ over~\eqref{eq:PresE} with side labels $U$ and $V$. The $a$--bigons contribute the empty word to the side labels. Each non-bigon $a$--face has side labels $z$ and $z^{-1}$. Therefore, the two side labels of $\Sigma$ are related by the automorphism of $F(B)$ which sends $z$ to $z^{-1}$ and fixes the remaining letters of $B$. This automorphism is an isometry with respect to the word metric induced by $B$. Hence $|U|_B=|V|_B$.

Now let $U$ and $V$ be $b$--conjugate words over $B^{\pm1}$, and let $\Sigma$ be the corresponding $b$--band. Again, the $b$--bigons contribute the empty word to the side labels. Each non-bigon $b$--face has the same $B$--label, namely $C_i$ for some $i\in\mathbb N$, on its two sides. Thus, after orienting the sides of $\Sigma$ compatibly, both side labels are products of the same sequence of words $C_i^{\pm1}$. Hence, the two side labels are freely equal in $F(B)$, so $|U|_B=|V|_B$.

Thus~\eqref{eq:PresE} satisfies \hyperlink{A}{($\ddag$)}, and therefore also satisfies \hyperlink{A}{($\dag$)}.
\end{proof}
    
Therefore, the general results of Section~\ref{sec:bilateral pres} apply to annular diagrams over~\eqref{eq:PresE}. In particular, since~\eqref{eq:PresE} satisfies \hyperlink{A}{($\ddag$)}, Proposition~\ref{Prop:embed} implies that the base group $F(B)$ embeds isometrically into $E$. From now on, we identify $F(B)$ with its image in $E$. For the remainder of this section, $|\cdot |$ denotes word length in $E$ with respect to the generating set $B \cup \{a,b\}$. Since $F(B)$ embeds isometrically in $E$, we have $|W| = |W|_B$ for every word $W$ over $B^{\pm 1}$.

 We restrict attention to $\{a,b\}$--minimal diagrams, that is, diagrams containing the minimum possible total number of $a$-- and $b$--faces among all diagrams with the same boundary labels and homotopy type, see Definition~\ref{def: T minimal van Kampen diagrams}.

 We call a diagram $\Delta$ \emph{strongly $\{a,b\}$--minimal} if it has the minimum possible total number of $a$-- and $b$--faces among all diagrams of the same homotopy type with the same boundary labels, and subject to this condition, has the minimum possible number of non--bigon $b$--faces. As the name suggests, every strongly $\{a,b\}$--minimal diagram is $\{a,b\}$--minimal, but not vice versa. For example, see the bottom left diagram in Figure~\ref{fig:shared-r-edge}, which is $\{a,b\}$-minimal but not strongly $\{a,b\}$-minimal. Indeed, it may be replaced by the diagram on the lower right, which consists of two $b$-bigons.

 % \gil{Maybe (of course, not necessary) add something like: as the name suggests, every strongly {a,b}-minimal diagram is {a,b}-minimal, but not vice versa. For example, see the diagram in figure ** (include a figure), which is {a,b}-minimal but not strongly {a,b}-minimal. Indeed, it may be replaced by... which consists of 2 $b$-bigons.}

\begin{lem} \label{lem:shared-r- edge}
   Let $\Delta$ be a strongly $\{a,b\}$--minimal annular diagram over \eqref{eq:PresE} whose boundary labels are cyclically $\{a,b\}$--minimal. Let $\pi$ and $\pi'$ be $b$--faces sharing an $r$--edge $e_r$. Then:
   \begin{enumerate}
       \item[(a)] $\pi$ and $\pi'$ have different ranks.
       \item[(b)] If $\pi$ and $\pi'$ belong to the same maximal radial $b$--band, then the two occurrences of $e_r$ lie on opposite sides of that band.
   \end{enumerate}
\end{lem}
\begin{proof}
    Reading the boundaries of $\pi$ and $\pi'$ clockwise, write $\partial\pi = e_rq$ and $\partial\pi' = e_r^{-1}l$, and let $\Xi$ be the disk subdiagram bounded by $ql.$ Suppose that $\pi$ and $\pi'$ have the same rank $i.$ Since the boundary label of each rank $i$ face is a cyclic shift of $[b, C_i]^{\pm 1}$, there are, up to inversion and cyclic shift, two possible ways to identify their $r$--edges. 

     % Depending on which side of each face contains $e_r$, one of the two configurations in Figure~\ref{fig:shared-r-edge} occurs. 

    In the first case, $\Lab(q)\equiv\Lab(l^{-1})$, so $\Xi$ may be removed and $q$ identified with $l^{-1}$. This contradicts $\{a,b\}$--minimality. In the second case, $\Lab(ql)$ freely reduces, up to cyclic shift and inversion, to $ b^2C_ib^{-2}C_i^{-1}.$ Replacing $\Xi$ by two $b$--bigons whose base vertices are joined by a path labeled $C_i$ preserves the number of $\{a,b\}$--faces and removes two ranked $b$--faces. This contradicts strong $\{a,b\}$--minimality. Therefore, $\pi$ and $\pi'$ have different ranks.
    
    For part~(b), first note that two $b$--bigons cannot share a $b$--edge, since removing them and identifying their remaining $b$--edges would produce a diagram with the same boundary labels and fewer $\{a,b\}$--faces. Hence two successive ranked $b$--faces in a maximal $b$--band either share a $b$--edge or are separated by one $b$--bigon. 
    
    Suppose now that $\pi$ and $\pi'$ belong to a maximal $b$--band $\Sigma$ and share an $r$--edge $e_r$, both occurrences of which lie on the same side of $\Sigma$. Choose $\pi$ and $\pi'$ to be an innermost such pair. After reversing the orientation if necessary, that side contains a subpath $e_rue_r^{-1}$. 

    Let $\Gamma$ be the subgraph traced by $u.$ We claim that $\Gamma$ is a tree. Otherwise, $\Gamma$ contains a simple cycle $v.$ Since $\Gamma$ is traced by a subpath of a side of $\Sigma,$ the cycle $v$ contains no $a$-- or $b$--edges. If $v$ separated the boundary components of $\Delta,$ then the terminal $b$--subband of $\Sigma$ in the annular subdiagram between $v$ and $\partial_{int}\Delta$ would have a $b$--edge on $v,$ which is impossible. Hence $v$ bounds a disk subdiagram $\Xi.$ Any $a$-- or $b$--face in $\Xi$ would belong to a maximal band with no end on $\partial\Xi$, and hence to a $0$--homotopic annulus, contrary to Lemma~\ref{Lem:0-hom}. Thus $\Xi$ contains no faces and is a tree, contradicting the existence of a cycle $v.$ Therefore, $\Gamma$ is a tree. 

    % Note that the loop $u$ is $0$--homotopic in $\Delta.$ Otherwise, $u$ would surround the central hole of $\Delta.$ Since $u$ lies on a side of $\Sigma$, its label contains no $a$-- or $b$--letters. Since $\Sigma$ is radial, a maximal $b$--subband of $\Sigma$ lying on the inner side of $u$ would then have an end $b$--edge on $u,$ which is impossible since $u$ contains no $b$--edges. Hence $u$ bounds a disk subdiagram $\Xi.$ Moreover, no maximal $a$-- or $b$-- band in $\Xi$ meets $\partial \Xi.$ Each would therefore be a $0$--homotopic annulus, contrary to Lemma~\ref{Lem:0-hom}. Consequently, $\Xi$ is a diagram over $F(B)$, so it contains no faces. Thus $\Xi$ is a tree. 
    
    If $u$ contained an $r$--edge, then, as a closed path in this tree, it would traverse the same edge again in the opposite direction. An innermost such pair of occurrences would correspond to two ranked $b$--faces sharing an $r$--edge on the same side of $\Sigma$. The subpath between them is properly contained in $u,$ contradicting the choice of $\pi$ and $\pi'.$ Therefore $u$ contains no $r$--edges. Every ranked $b$--face strictly between $\pi$ and $\pi'$  contributes an $r$--edge to this side of $\Sigma.$ Hence there are no such faces, so $\pi$ and $\pi'$ are successive ranked faces of $\Sigma$. They therefore either share a $b$--edge or are separated by one $b$--bigon.

    Let $\pi$ and $\pi'$ have ranks $i$ and $j$ respectively, and let $u_1$ and $u_1'$ be the portions of the sides from the corresponding occurrences of $e_r$ to the facing $b$--edges. The intervening $b$--bigon, if present, contributes an empty side path, so $u = u_1u_1'^{-1}.$ Since $u$ is a closed path in the tree $\Gamma,$   Lemma~\ref{Lem:van Kampen}(a) gives $\Lab(u_1)=_{F(B)}\Lab(u_1').$ Up to inversion, $\Lab(u_1)\equiv z^{-f(i)}x^{ i}$ or $x^{i}z^{f(i)}$, and $\Lab(u_1')\equiv z^{-f(j)}x^{j}$ or $x^{ j}z^{f(j)}$.  Comparing freely reduced words gives $i=j$, contradicting part~(a).
\end{proof}

As in Section~\ref{sec:proof of main1}, given an annular diagram $\Delta$, we denote the total length of its boundary components by $\ell(\partial \Delta) = \ell(\partial_{ext}\Delta) + \ell(\partial_{int}\Delta).$

\begin{lem}\label{lem: radial b bands}
Let $\Delta$ be a strongly $\{a,b\}$--minimal annular diagram over \eqref{eq:PresE} whose boundary labels are cyclically
$\{a,b\}$--minimal. Suppose that $\Delta$ contains a radial $b$--band. Then there exists a path $m$ connecting the two boundary components of $\Delta$, such that $$|\Lab(m)| \leq \ell(\partial \Delta) + 1.$$
\end{lem}

\begin{figure}
 \centering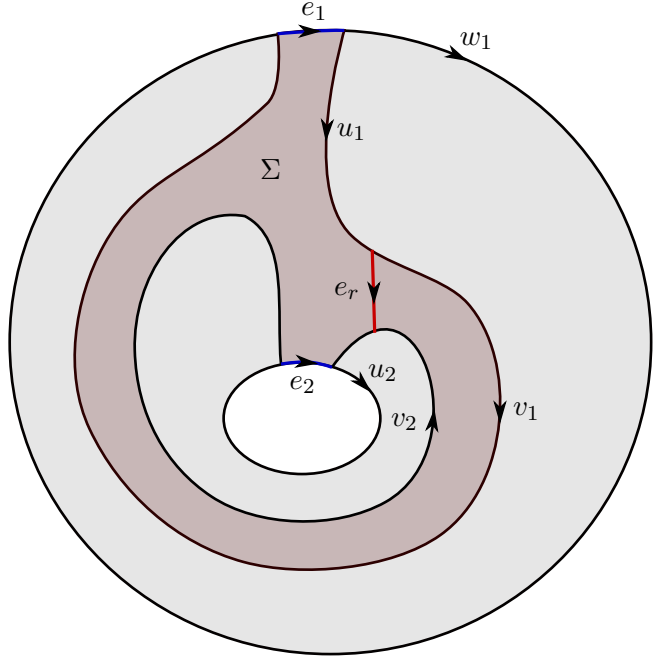\\
  \caption{Case 1 of Lemma \ref{lem: radial b bands}}
\end{figure}

\begin{proof}
Fix a radial $b$--band $\Sigma$ in $\Delta$, and let $u$ be one of its sides. If $u$ contains no $r$--edges, then every face of $\Sigma$ is a $b$--bigon. Hence $u$ has length zero and joins the two boundary components of $\Delta,$ so we may take $m = u.$ Assume, therefore, that $u$ contains an $r$--edge.

 \emph{Case 1. Some $r$--edge on $u$ does not belong to $\partial\Delta.$} 
 Let $e_r$ be such an edge, and let $\pi$ be the face of $\Sigma$ incident to $e_r.$ Since $e_r$ is an interior edge, it is also incident to a second ranked $b$--face $\pi'.$ Suppose first that $\pi'\notin\Sigma$, and let $\Sigma'$ be the maximal $b$--band containing $\pi'.$ Since $\Delta$ contains the radial $b$--band $\Sigma$ and its boundary labels are cyclically $\{a,b\}$--minimal, Lemma~\ref{Lem:NoArches} implies that $\Sigma'$ is radial.

 Now suppose that $\pi' \in \Sigma.$ By Lemma~\ref{lem:shared-r- edge}(b), the two occurrences of $e_r$ lie on opposite sides of $\Sigma.$ In this case, set $\Sigma'=\Sigma.$

 Thus, in either case, there is a side $u'$ of the radial $b$--band $\Sigma'$ on which $e_r$ occurs with the opposite orientation. Write
\[
u=u_1e_ru_2
\quad\text{and}\quad
u'=u_2'e_r^{-1}u_1'.
\]
Let $w_1$ and $w_2$ be the subpaths of the two boundary components of $\Delta$ such that the loops
\[
u_1u_1'w_1^{-1}
\quad\text{and}\quad
u_2'u_2w_2^{-1}
\]
bound disk subdiagrams. Hence
\[
\Lab(u_1u_1')=_E\Lab(w_1)
\quad\text{and}\quad
\Lab(u_2'u_2)=_E\Lab(w_2).
\]

Let $i$ and $j$ be the ranks of $\pi$ and $\pi'$, respectively. By Lemma~\ref{lem:shared-r- edge}(a), $i \neq j.$  Since the side labels of $b$--bands are products of  words $C_k^{\pm1}$, the factors meeting in  $u_1u_1'$ and $u_2'u_2$  have exactly the same form as those considered in Case~1 of Lemma~\ref{lem: linear m connecting the boundary}. The cancellation estimates proved there therefore give
\[
|\Lab(u_1)|\leq|\Lab(u_1u_1')|
\quad\text{and}\quad
|\Lab(u_2)|\leq|\Lab(u_2'u_2)|.
\]
Since $F(B)$ embeds isometrically in $E$, it follows that
\begin{align*}
|\Lab(u)|-1
    &=|\Lab(u_1)|+|\Lab(u_2)|\\
    &\leq|\Lab(u_1u_1')|+|\Lab(u_2'u_2)|\\
    &=|\Lab(w_1)|+|\Lab(w_2)|\\
    &\leq\ell(w_1)+\ell(w_2)\\
    &\leq\ell(\partial\Delta).
\end{align*}
Therefore,
\[
|\Lab(u)|\leq\ell(\partial\Delta)+1.
\]
Taking $m=u$ proves the assertion in this case.
 
\emph{Case 2. Every $r$--edge of $u$ belongs to $\partial\Delta$.}
 Note that $u$ has exactly the same form as the $C$--side of the radial $t$--band considered in Case~2 of Lemma~\ref{lem: linear m connecting the boundary}. The argument from that case therefore applies without change and produces a subpath $m$ of $u$ joining the two boundary components of $\Delta$ such that $|\Lab(m)|\leq \ell(\partial\Delta).$
\end{proof}

We now prove the upper bound for the conjugator length function of $E$.
 \begin{lem}[Upper bound]\label{lem:E upper bound} $\CL_E(n) \preceq n.$
\end{lem}
\begin{proof} Let $g,h\in E$ be conjugate elements of total word length $|g| + |h| \leq n$ for some $n \in \mathbb{N}.$ By Lemma~\ref{Lem:CTM}, there exist cyclically $\{a,b\}$--minimal words $U,V$ over $(B\cup\{a,b\})^{\pm1}$ representing elements $g'\in g^E$ and $h'\in h^E$, respectively, such that 
\[\|U\|\leq |g|, c(g,g')\leq |g|, \text{ and } \|V\|\leq |h|, c(h,h')\leq |h|.\]

Since $c(g,g')\leq |g|$ and $c(h',h)=c(h,h')\leq |h|$, it is enough to prove that $c(g',h')$ is bounded above linearly in $n$. Let $\Delta$ be a \textcolor{teal}{strongly} $\{a,b\}$--minimal annular diagram with boundary labels $U$ and $V$; such a diagram exists by Lemma~\ref{Lem:van Kampen}(b). 

\emph{Case 1. Neither $U$ nor $V$ contains a $b$--letter.} Suppose that $\Delta$ contains a $b$--face, and let $\Sigma$ be a maximal $b$--band containing it. Since neither boundary component contains a $b$--edge, $\Sigma$ is a $b$--annulus. Lemma~\ref{Lem:0-hom} rules out $0$--homotopic $b$--annuli, so $\Sigma$ is concentric. Let $u$ and $v$ be its outer and inner sides, oriented compatibly. Then $\Lab(u) =_{F(B)} \Lab(v).$ By Lemma ~\ref{Lem:van Kampen}(b), there is an annular diagram over $F(B)$ with boundary labels $\Lab(u)$ and $\Lab(v).$ Replacing $\Sigma$ by this diagram produces an annular diagram with the same boundary labels $U$ and $V$ and fewer $b$--faces, contradicting $\{a,b\}$--minimality. Thus $\Delta$ contains no $b$--faces.
We may regard \eqref{eq:PresE} as a $b$--bilateral presentation over
    \[
            H:=\langle x,r,z,a \mid a^2=1,\ z^a=z^{-1}\rangle
    .\]
   The proof of Lemma~\ref{lem: E ddag}  shows that~\eqref{eq:PresE} satisfies \hyperlink{A}{($\dag$)} and \hyperlink{A}{($\ddag$)} as a $b$--bilateral presentation over $H$. Proposition~\ref{Prop:embed} therefore implies that the natural map $H\to E$ is an isometric embedding, and we identify $H$ with its image in $E$.

   Since $\Delta$ contains no $b$--faces, it is an annular diagram over the above presentation of $H.$  In particular, $g'$ and $h'$ are conjugate in $H$. Observe that $H \cong F_2 * D_\infty,$ which is a hyperbolic group. It is well known that hyperbolic groups have linear conjugator length function;see for example \cite[Chapter~III.$\Gamma$, Lemma~2.9]{BH99}.  Hence there are constants $C,D \geq 0$ such that \[c_H(g',h') \leq C(|g'|_H+|h'|_H)+D.\]
   Since $U$ and $V$ represent $g'$ and $h'$, respectively, and $H$ embeds isometrically in $E,$ we obtain 
   \begin{align*} c_E(g',h') &\leq c_H(g',h')\\
                           &\leq C(|g'|_H+|h'|_H)+D\\ 
                           &\leq C(\|U\|+\|V\|)+D\\ 
                           &\leq 2Cn+D. 
    \end{align*} 
   
   Thus, $c_E(g',h')$ is bounded above linearly in $n.$

\emph{Case 2. At least one of $U$ and $V$ contains a $b$--letter.}
Without loss of generality, suppose that $U$ contains a $b$--letter, and let $u$ denote the boundary component labeled by $U$. If the boundary components of $\Delta$ intersect, choose a common vertex as the initial vertex of both boundary components. The resulting labels are cyclic shifts $U'$ and $V'$ of $U$ and $V$ representing the same element of $E$. Hence $c(g',h')\leq \|U\|+\|V\|\leq n$. We may therefore assume that the boundary components are disjoint.

Suppose that a $b$--edge $e_b$ on $u$ is incident to no face. Since the two boundary components of $\Delta$ are disjoint, $e_b$ is a cut edge of the underlying planar complex. Indeed, otherwise $e_b$ would lie on a simple cycle and create an additional hole in $\Delta.$ Removing the interior of $e_b$ separates $\Delta$ into two components, exactly one of which contains the other boundary component. The excursion of $u$ through the other component is therefore $0$--homotopic. After changing the initial vertex of $u,$ write $u = pq$, where $p$ is this excursion and traverses $e_b$ in both directions. Then $\Lab(p) =_E 1$ and $\Lab(q) =_E \Lab(u)$, while $\Lab(q)$ contains strictly fewer $b$-- letters than $\Lab(u).$ This contradicts the cyclic $\{a,b\}$--minimality of $\Lab(u).$

   So, let $\Sigma$ be the maximal $b$--band containing the $b$--face incident to $e_b.$ Since $\Sigma$ meets $\partial \Delta$ over a $b$--edge it is not a $b$--annulus, while Lemma~\ref{Lem:NoArches} rules out the possibility that it is an arch. Hence $\Sigma$ is radial. By Lemma~\ref{lem: radial b bands}, there exists a path $m$ connecting the two boundary components of $\Delta$ such that 
   \[|\Lab(m)|\leq \ell(\partial\Delta)+1.\]
   Taking the endpoints of $m$ as the initial vertices of the boundary components gives cyclic shifts $U'$ and $V'$ of $U$ and $V$ such that $\Lab(m)^{-1}U'\Lab(m)=_E V'.$
   Since $U'$ and $V'$ represent conjugates of $g'$ and $h'$ respectively, by conjugator of length at most $\|U\|$ and $\|V\|,$ we obtain 
   \begin{align*}
    c(g',h')
    &\leq |\Lab(m)|+\|U\|+\|V\| \\
    &\leq 2\ell(\partial\Delta)+1 \\
    &=2(\|U\|+\|V\|)+1 \\
    &\leq 2n+1.
\end{align*}
Thus, in both cases, $c(g',h')$ is bounded above linearly in $n$. Consequently, $c(g,h)$ is bounded above linearly in $n$, and hence $\CL_E(n)\preceq n$.
\end{proof}

\begin{proof}[Proof of Theorem~\ref{Thm:main2}]
  
    Let $f,g: \mathbb{N} \to \mathbb{N} \cup \{0\}$ be non-decreasing, at least linear functions.  After replacing them by equivalent functions if necessary, we may assume that they are strictly increasing.

    For each $h \in \{f, g\}$, let $G_h$ be the group defined by \eqref{Eq:PresG main1}, and let $E_h$ be its index-two extension defined by \eqref{eq:PresE}. By Theorem~\ref{Thm:main1}, $\CL_{G_h}\sim h$, while Lemma~\ref{lem:E upper bound} gives $\CL_{E_h}\preceq n$. Since $E_h$ contains a nonabelian free subgroup, its center has infinite index - since otherwise, $F(B) \cap Z(E_h)$ would have finite index in $F(B)$, whereas $F(B) \cap Z(E_h) \leq Z(F(B)) = 1.$ Proposition~\ref{Prop:BCL} therefore gives $n\preceq \CL_{E_h}$. Thus,
    $\CL_{G_h} \sim h$ and $\CL_{E_h} \sim n$ for $h \in \{f, g\}.$ We also use the elementary fact $\CL_{A \times B} \sim \max\{CL_A, CL_B\}$ for finitely generated groups $A$ and $B.$ To see this, choose finite generating sets $X$ and $Y$ for $A$ and $B$, so $U = (X \times \{1\}) \cup (\{1\} \times Y)$ is a finite generating set for $A \times B.$ 
    Since conjugation is coordinatewise and $|(u,v)|_U = |u|_X + |v|_Y,$ we have that
\[
c_{A\times B}\bigl((a,b),(a',b')\bigr)
=c_A(a,a')+c_B(b,b')
\]
for every conjugate pair. The natural embeddings of $A$ and $B$ into $A\times B$ give the lower bound, while the preceding formula gives the upper bound:
\[
\max\{CL_A(n),\CL_B(n)\}
\leq \CL_{A\times B}(n)
\leq \CL_A(n)+\CL_B(n).
\]
This proves the estimate.

Set $S_1=G_f\times E_g$ and $S_2=E_f\times G_g$. Their intersection in $E_f\times E_g$ is $G_f\times G_g$, which has index two in each. Hence, $S_1$ and $S_2$ are commensurable. Finally, since $f$ and $g$ are at least linear, the product estimate gives $\CL_{S_1}\sim\max\{\CL_{G_f},\CL_{E_g}\}\sim f$ and $\CL_{S_2}\sim\max\{\CL_{E_f},\CL_{G_g}\}\sim g$.
\end{proof}

\section{Groups with bounded conjugator length}\label{sec:bounded CL}

In this section, we give an algebraic characterization of groups with bounded conjugator length. 

\begin{prop}\label{Prop:BCL}
Let $G$ be a finitely generated group. 
\begin{enumerate}
    \item[(a)] $\CL_G$ is bounded if and only if $|G/Z(G)|<\infty$.
    \item[(b)] If $\CL_G$ is unbounded, then it is at least linear.
\end{enumerate}
\end{prop}

\begin{proof}
We fix a finite generating set $X$ of $G$. Suppose first that $|G/Z(G)|<\infty$. Let $T$ be a finite transversal of $Z(G)$ in $G$. If some elements $a, b\in G$ are conjugate by an element $g\in G$, then $g=tz$ for some $t\in T$ and $z\in Z(G)$. Since $z$ is central, we have $t^{-1}at=g^{-1}ag=b$. This implies that $\CL_{G,X}(n)\le \max\{|t|_X \mid t\in T\}$. 

To prove the ``only if" direction in (a), we first note that if the conjugacy class $x^G$ of every $x\in X$ is finite, then $|G:C_G(x)|=|x^G|<\infty$ for all $x\in X$; hence, the index of $$Z(G)=\bigcap_{x\in X} C_G(x)$$ is also finite. Thus, the assumption $|G/Z(G)|=\infty$ implies the existence of $x\in X$ such that the conjugacy class $x^G$ is infinite.  In particular, for every $n\in \NN$, there exists $g\in G$ of length $|g|_X\ge n$ such that $g$ is conjugate to $x$. 

Suppose that $x^t=g$ for some $t\in G$ and let $w$ be a word over the alphabet $X^{\pm 1}$ representing $t$ in $G$. For every $0\le i\le k$, where $k=\| w\|$, let $t_i$ denote the element of $G$ represented by the prefix of $w$ of length $i$. For all $1\le i\le k$, we have $t_i=t_{i-1}y$ for some $y\in X$ and 
$$
\big| |x^{t_i}|_X -|x^{t_{i-1}}|_X\big|=\big| |y^{-1}x^{t_{i-1}}y|_X -|x^{t_{i-1}}|_X\big|\le 2.
$$
Since $|x^{t_0}|_X=|x|_X=1$ and $|x^{t_{k}}|_X=|x^t|_X=|g|_X\ge n$, there exists an index $j$ such that 
\begin{equation}\label{Eq:|h|}
n-2\le |x^{t_j}|_X \le n.
\end{equation}
Note that if $x$ is conjugate to the element $h=x^{t_j}$ by some $s\in G$, then 
$$
|h|_X=|x^s|_X\le 1+2|s|_X.
$$
Combining with the left inequality in (\ref{Eq:|h|}), we obtain $|s|_X\ge (n-3)/2$. By the right inequality in (\ref{Eq:|h|}), we have $|x|_X+|h|_X\le n+1$. Thus,  $\CL_{G,X}(n+1)\ge (n-3)/2$ for all $n\in \NN$. In particular, $\CL_G$ is unbounded.

To prove (b), we only need to note that if $\CL_G$ is unbounded, then $|G/Z(G)|=\infty $ by (a), and under this assumption, the inequality $\CL_{G,X}(n)\succcurlyeq n$ is established above.
\end{proof}

\begin{rem}
    A finitely generated group $G$ satisfies $|G/Z(G)|<\infty $ if and only if $|[G,G]|<\infty$. The backward implication is obvious, and the forward implication is the well-known Schur theorem. 
\end{rem}

Proposition \ref{Prop:BCL} allows us to establish some form of quasi-isometric rigidity for groups with bounded conjugator length functions.

\begin{cor}\label{Cor:BCL}
    Let $G$ be a finitely generated group with bounded conjugator length function. If $H$ is a finitely generated group quasi-isometric to $G$, then $\CL_H\preccurlyeq n$.
\end{cor}
\begin{proof}
    By Proposition \ref{Prop:BCL}, we have $|G/Z(G)|<\infty$. It is well-known and easy to derive from the combination of Gromov's polynomial growth theorem \cite{Gro81} and Bass' growth formula for nilpotent groups \cite{Bass} that the class of finitely generated virtually abelian groups is quasi-isometrically rigid (for an alternative argument using asymptotic cones, see Section 4.3 of \cite{Dru}). Thus, $H$ is also virtually abelian. For such groups, the inequality $\CL_H\preccurlyeq n$ is easy to prove (see \cite[Corollary 2.3.19]{S15a}).
\end{proof}

\begin{rem}
The inequality $\CL_H\preccurlyeq n$ in Corollary \ref{Cor:BCL} cannot be improved. Indeed, the infinite dihedral group $D_\infty =\mathbb Z \rtimes \ZZ_2$ is quasi-isometric to $\ZZ$ and has linear conjugator length function (e.g., by the combination of Proposition \ref{Prop:BCL} and \cite[Corollary 2.3.19]{S15a}).
\end{rem}

\end{document}